\documentclass[11pt,leqno,fleqn]{amsart}
\usepackage{amsmath,amstext,amsthm,amssymb,amsxtra,bbm,enumerate}
\usepackage{xcolor}
\usepackage[top=2.5cm, bottom=2.5cm, left=2.5cm, right=2.5cm]	{geometry}
\usepackage{pxfonts} 
\usepackage[T1]{fontenc}
\usepackage{lmodern}

\usepackage{euler} 
\usepackage{scalerel}

\usepackage{mathtools}
\mathtoolsset{showonlyrefs,showmanualtags}

\DeclareMathAlphabet\gothic{U}{euf}{m}{n}

\makeatletter
\def\eqnarray{\stepcounter{equation}\let\@currentlabel=\theequation
\global\@eqnswtrue
\tabskip\@centering\let\\=\@eqncr
$$\halign to \displaywidth\bgroup\hfil\global\@eqcnt\z@
$\displaystyle\tabskip\z@{##}$&\global\@eqcnt\@ne
\hfil$\displaystyle{{}##{}}$\hfil
&\global\@eqcnt\tw@ $\displaystyle{##}$\hfil
\tabskip\@centering&\llap{##}\tabskip\z@\cr}

\def\endeqnarray{\@@eqncr\egroup
\global\advance\c@equation\m@ne$$\global\@ignoretrue}

\def\@yeqncr{\@ifnextchar [{\@xeqncr}{\@xeqncr[5pt]}}
\makeatother

\allowdisplaybreaks[1]

\newtheorem{lemm}{Lemma}[section]
\newtheorem{thrm}[lemm]{Theorem}
\newtheorem{coro}[lemm]{Corollary}
\newtheorem{eeg}[lemm]{Example}
\newtheorem{rrema}[lemm]{Remark}
\newtheorem{prop}[lemm]{Proposition}
\newtheorem{ddefi}[lemm]{Definition}

\newcommand{\supp}{\mathrm{supp\,}}

\newcommand{\ls}{\lesssim}

\newcommand{\D}{\partial}
\newcommand{\gota}{\gothic{a}}
\newcommand{\di}{\mathrm{div}}

\newcommand{\lesi}{\lesssim}
\newcommand{\f}{\frac}
\newcommand{\vc}{\infty}

\newcounter{teller}

\newenvironment{tabel}{\begin{list}%
{\rm  (\alph{teller})\hfill}{\usecounter{teller} \leftmargin=1.1cm
\labelwidth=1.1cm \labelsep=0cm \parsep=0cm}
          }{\end{list}}

\newcounter{tellerr}

\newenvironment{tabeleq}{\begin{list}%
{\rm  (\roman{tellerr})\hfill}{\usecounter{tellerr} \leftmargin=1.1cm
\labelwidth=1.1cm \labelsep=0cm \parsep=0cm}
             }{\end{list}}

\newcounter{tellerrr}

\newcounter{proofstep}

\newcommand{\Ni}{\mathbb{N}}

\newcommand{\Ri}{\mathbb{R}}
\newcommand{\Ci}{\mathbb{C}}

\newcommand{\loc}{\mathrm{loc}}

\newcommand{\ca}{\quad \mbox{and} \quad}

\newlength{\hightcharacter}
\newlength{\widthcharacter}

\makeatletter
\renewcommand*\env@matrix[1][*\c@MaxMatrixCols c]{%
\hskip -\arraycolsep
\let\@ifnextchar\new@ifnextchar
\array{#1}}
\makeatother

\def\barint{\kern4pt
	\raise3.4pt\hbox{\vrule height.8pt width7pt}%
	\kern-8.0pt 
	\int}

\title{Heat kernels of generalized degenerate Schr\"odinger operators and Hardy spaces}

\author{The Anh Bui}
\address{Department of Mathematics of Statistics, Macquarie University, NSW 2109,
	Australia}
\email{the.bui@mq.edu.au, bt\_anh80@yahoo.com}

\author{Tan Duc Do}
\address{Vietnamese-German University, Binh Duong, Vietnam}
\email{tan.dd@vgu.edu.vn}

\author{Nguyen Ngoc Trong}
\address{Ho Chi Minh City University of Education, Ho Chi Minh City, Vietnam}
\email{trongnn@hcmue.edu.vn}

\thanks{2010 {\it Mathematics Subject Classification}: 47F05, 65M80, 35K08, 42B30}
\thanks{{\it Key words and phrases}: generalized degenerate Schr\"odinger operator, fundamental solution, heat kernel, Hardy space.}

\begin{document}
\begin{abstract}
Let $\displaystyle L = -\frac{1}{w} \, \di(A \, \nabla u) + \mu$	be the generalized degenerate Schr\"odinger operator in $L^2_w(\Ri^d)$ with $d\ge 3$ with suitable weight $w$ and measure $\mu$. The main aim of this paper is threefold. First, we obtain an upper bound for the fundamental solution of the operator $L$. Secondly, we prove some estimates for the heat kernel of $L$ including an upper bound, the H\"older continuity and a comparison estimate. Finally, we apply the results to study the maximal function characterization for the Hardy spaces associated to the critical function generated by the operator $L$.
\end{abstract}

\maketitle

\setcounter{page}{1}

\section{Introduction}

Consider the generalized degenerate Schr\"odinger operator of the form 
\begin{equation}
\label{eq-L}
L = -\frac{1}{w} \, \di(A \, \nabla u) + \mu
\end{equation}
in $L^2_w(\Ri^d)$ with $d\ge 3$. Here $w$, $A$ and $\mu$ satisfy the following conditions:

\begin{itemize}
\item The coefficient matrix $A$ is real symmetric with measurable entries.
Furthermore there exists a constant $\Lambda \ge 1$ such that 
\begin{equation} \label{degenerate}
\hskip 3cm\Lambda^{-1} \, w(x) \, |\xi|^2 \leq A(x)\xi \cdot \overline{\xi} \leq \Lambda \, w(x) \, |\xi|^2
\end{equation}
for a.e.  $x \in \Ri^d$ and for all $\xi \in \Ci^d$.

\item The weight $w \in A_2$. See below for the precise definition of the class $A_2$. Moreover, we also assume that  $w \in RD_\beta$ with some $\beta>2$, i.e., there exists a $C > 0$ such that 
\begin{equation} \label{(RD)}
\hskip 3cm w(B(x,tr)) \ge C \, t^\beta \, w(B(x,r))
\end{equation}
for all $t > 1$ and $x\in \Ri^d$.

\item $\mu$ is a positive Radon measure satisfying the following conditions:
	\begin{enumerate}[{\rm (i)}]
		\item There exist $C_0>0$ and $\delta>0$ such that 
		\begin{equation}\label{(M1)}
	\hskip 3cm 	\frac{r^2}{w(B(x,r))} \, \pi(B(x,r)) 
		\le C_0 \, \left( \frac{r}{R} \right)^{\delta} \, \frac{R^2}{w(B(x,R))} \, \pi(B(x,R))
		\end{equation}
		for all $x \in \Ri^d$ and $R>r> 0$, where $d\pi = w \, d\mu$.
		
		\item There exists a $C_1 > 0$ such that
		\begin{equation}\label{(M2)}
	\hskip 3cm 	\pi(B(x,2r)) 
		\le C_1 \, \left( \pi(B(x,r)) + \frac{w(B(x,r))}{r^2} \right)
		\end{equation}
		for all $x \in \Ri^d$ and $r > 0$, where $d\pi = w \, d\mu$.
	\end{enumerate}
\end{itemize}
A precise description of $L$ via form method is given in Section \ref{operator def}. 

Let  $w\in L^1_{{\rm loc}}(\Ri^d)$ be a nonnegative locally integrable function. For $p\in(1,\vc)$, we say that $w\in A_p$ if 
\[
\left( \frac{1}{|B|} \, \int_B w(x) \, dx \right) \, 
\left( \frac{1}{|B|} \, \int_B w(x)^{1-p'} \, dx \right)^{p-1}
\le C
\]
for all balls $B\subset \Ri^d$.

Note that if $w\in A_p$ then 
\begin{equation}\label{doubling from w}
w(\lambda B)\lesi \lambda^{dp}w(B)
\end{equation}
for all balls $B$ and $\lambda\ge 1$.

We would like to summarize the body of research regarding the generalized degenerate Schr\"odinger operator of the form \eqref{eq-L}.
\begin{enumerate}[{\rm (a)}]
	\item The well-known form of the operator $L$ is when $w=1$, $A(x)=I$ and $\mu(x) =V(x)dx$ with $V$ satisfying a reverse H\"older's inequality, i.e.,
	\[
	\Big(\f{1}{|B|}V(x)^qdx\Big)^{1/q}\lesi \f{1}{|B|}V(x)dx
	\]
	for all balls $B$ with $q\ge d/2$. Such Schr\"odinger operators were introduced in \cite{Shen} of which the estimates for the fundamental solution of $L$ and the Riesz transforms were investigated. The theory of new Hardy spaces associated to the operator $L$ was treated in \cite{DZ1, DZ2, BDK1}.  
	\item The degenerate elliptic operators corresponding to $L$ of the form \eqref{eq-L} with $\mu=0$ were studied in \cite{FKS} in which Fabes, Kenig and  Serapioni proved some results on local regularity of degenerate elliptic operators on domains. The Hardy spaces $H^1_L$ associated to the degenerate Schr\"odinger with $\mu(x)=V(x)dw(x)$, where $V(x)$ satisfies a reverse H\"older's inequality were obtained in \cite{Dzi}.
	
	\item The study of Schr\"odinger operators of the form \eqref{eq-L} with $\mu$ being a Radon measure is less well-known. The case $w=1$ and $A=I$ was first introduced by Shen \cite{Shen}. Note that in this case the measure $\mu$ contains the class of measures satisfying scale--invariant  Kato condition. See \cite[p.522]{Shen}. In this work, Shen proved the bounds for the fundamental solutions of $L$ and the boundedness of the Riesz transforms and the imaginary powers of $L$. The Hardy spaces $H^1_L$ related to these operators have been studied recently in \cite{WY}.    
\end{enumerate}

Motivated by the above body of work, this paper will work with the generalized degenerate Schr\"odinger operators of the form \eqref{eq-L}. We would like to describe our main results.
\begin{enumerate}[{\rm (a)}]
	\item First, we prove an upper bound for the fundamental solution to the operator $L$. See Theorem \ref{main}. In a particular case when $w=1$ and $A=I$ our result is in line with that in \cite{Shen}. Note that our method can be modified to obtain the exponential decay in the upper bound, but we do not aim to pursue this problem since the polynomial decay in the upper bound is enough for our purpose.
	\item Secondly, we obtain some estimates for the heat kernel of the semigroup $e^{-tL}$ generated by $L$. The results include the upper bound, the H\"older continuity estimates and the comparison estimates. See Theorem \ref{main 2}. It emphasizes that our estimates recover known estimates in the corresponding particular case of $L$ such as Schr\"odinger operator on $\Ri^d$ (see \cite{DZ1}), degenerate Schr\"odinger operator (see \cite{Dzi}) and the generalized  Schr\"odinger operator on $\Ri^d$ (see \cite{WY}). 
	\item The last result relates to the theory of Hardy spaces associated to differential operators. The theory of Hardy spaces adapted to general operators was initially introduced by \cite{ADM}. Then it  has become an interesting topic in harmonic analysis and has attracted a great deal of attention. See for example \cite{DY, HM, HLMMY} and the references therein. In the last result, we   apply the findings on the heat kernel of $L$ to prove the maximal function characterization to the Hardy spaces associated to $L$. Let us remind that the maximal function characterization of the Hardy spaces  has a long history. The maximal function characterizations for the classical Hardy spaces were obtained in \cite{BGS, CT}. Then the results were extended to new Hardy spaces $H^1_L$ associated to Schr\"odinger operators $L$ in various settings. See \cite{DZ1, DZ2, YZ, WY, Dzi}. In Theorem \ref{mainthm 3} we prove the maximal function for the new Hardy spaces $H^p_L$ for $p\le 1$. We note that our result not only recovers known results in  \cite{DZ1, DZ2, YZ, WY, Dzi}, but also extends those to the range $p< 1$.    
\end{enumerate}

To formulate our main results the following notion plays a fundamental role.
For all $x \in \Ri^d$ define
\begin{equation}\label{critical function}
\rho_w(x,\mu) 
:= \frac{1}{m_w(x,\mu)} 
:= \sup \left\{ r>0: \frac{r^2}{w(B(x,r))} \, \pi(B(x,r)) \le C_1 \right\},
\end{equation}
where $C_1$ is given by \eqref{(M2)}, which is called the \emph{critical function}.

Our first main result on the upper bound of the fundamental solution of $L$ is as follows.

\begin{thrm} \label{main}
	Let $\Gamma_\mu(x,\cdot)$ be the fundamental solution of $L$.
	Then for every $k \in \Ni$ there exists a $C=C(k) > 0$ such that 
	\begin{equation}\label{upper bound for gammamu}
	0
	\le \Gamma_\mu(x,y)
	\le \frac{C}{\big( 1 + |x-y| \, m_w(x,\mu) \big)^k} \frac{|x-y|^2}{w(B(x,|x-y|))}
	\end{equation}
	for all $x, y \in \Ri^d$ such that $x \ne y$.
\end{thrm}

The next result will be on the heat kernel estimates of $L$. Before coming to details, we recall the principle part of $L$ given by
\[
L_0 u = -\frac{1}{w} \, \di( A \, \nabla u).
\]
Denote by $h_t(\cdot,\cdot)$ and $k_t(\cdot,\cdot)$ the kernels of $e^{-tL_0}$ and $e^{-tL}$ for $t>0$, respectively. Then we have the following.

\begin{thrm} \label{main 2}
Let $k_t(\cdot,\cdot)$ be the kernel of  $e^{-tL}$. Then we have:
\begin{enumerate}[{\rm (i)}]
	\item There exists a $c > 0$ such that for all $N \ge 0$ there exists a $C = C(N) > 0$ satisfying
	\[
	0
	\le k_t(x,y)
	\le \frac{C}{w(B(x,\sqrt{t}))} \, \exp\Big(-\f{|x-y|^2}{ct}\Big) 
	\left( 1 + \frac{\sqrt{t}}{\rho_w(x,\mu)}+ \frac{\sqrt{t}}{\rho_w(y,\mu)} \right)^{-N}
	\]
	for all $x, y \in \Ri^d$ and $t > 0$.
	
	\item For any $0<\theta<\min \{\delta,\gamma\}$, there exists a $C>0$ such that 
	\begin{equation}\label{Holder p_t(x,y)}
	|k_{t}(x,y)-k_{t}(x,\overline{y})|\leq   C \, \Big(\f{|y-\overline{y}|}{\sqrt{t}}\Big)^{\theta}\f{1}{w(B(x,\sqrt{t}))}\exp\Big(-\f{|x-y|^2}{ct}\Big)
	\end{equation}
	for all $t>0$ and $|y-\overline{y}|<\sqrt t$, where $\delta$ is the constant in \eqref{(M1)} and $\gamma$ is the constant in Proposition \ref{ht prop} (ii).
	
	\item There exist constants $C, c > 0$ such that 
		\begin{equation}\label{thm- difference 2 heat kernels}
		0 \le h_t(x,y) - k_t(x,y) \le C \, \frac{1}{w(B(x,\sqrt{t}))} \, \exp\Big(-\f{|x-y|^2}{ct}\Big)\, \left( \frac{\sqrt{t}}{\rho_w(x,\mu)} \right)^{\delta} 
		\end{equation}
		for all $t > 0$ and $x, y \in \Ri^d$.
\end{enumerate}

\end{thrm}

We now move on to the final main result regarding new Hardy spaces. In order to present this result, we introduce new local Hardy spaces associated to critical functions $\rho$. In what follows, for $w\in A_2$ we define $q_w=\inf\{p\in (1,\infty): w\in A_p\}$. It is well-known that $q_w<2$ (see \cite{St}), and we set $n=q_w d$. 
\begin{ddefi}\label{def: rho atoms}
	Let $\rho_w$ be the critical function defined by \eqref{critical function}. Let $p\in (\f{n}{n+1},1]$, $q\in [1,\vc]\cap (p,\vc]$ and $\epsilon\in (0,1]$. A function $a$ is called a  $(p,q,\rho_w,\epsilon)$-atom associated to the ball $B(x_0,r)$ if
	\begin{enumerate}[{\rm (i)}]
		\item ${\rm supp}\, a\subset B(x_0,r)$ and $r\le \rho_w(x_0,\mu)$;
		\item $\|a\|_{L^q_w(\Ri^d)}\leq w(B(x_0,r))^{1/q-1/p}$;
		\item $\displaystyle \int a(x)dw(x) =0$ if $r<\epsilon \rho_w(x_0,\mu)/4$.
	\end{enumerate}
\end{ddefi}
For the sake of convenience, when $\epsilon=1$ we shall write $(p,q,\rho_w)$ atom instead of $(p,q,\rho_w,\epsilon)$-atom.

\begin{ddefi}
	Let $\rho_w$ be the critical function defined by \eqref{critical function}. Let $p\in (\f{n}{n+1},1]$, $q\in [1,\vc]\cap (p,\vc]$ and $\epsilon\in (0,1]$. We  say that $f=\sum
	\lambda_ja_j$ is an atomic $(p,q,\rho_w,\epsilon)$-representation if
	$\{\lambda_j\}_{j=0}^\infty\in l^p$, each $a_j$ is a $(p,q,\rho_w,\epsilon)$-atom,
	and the sum converges in $L^2_w(\Ri^d)$. The space $h^{p,q}_{at,\rho_w,\epsilon}(\Ri^d,w)$ is then defined as the completion of
	\[
	\left\{f\in L^2_w(\Ri):f \ \text{has an atomic
		$(p,q,\rho_w,\epsilon)$-representation}\right\},
	\]
	with the norm given by
	$$
	\|f\|_{h^{p,q}_{at,\rho_w,\epsilon}(\Ri^d,w)}=\inf\Big\{\Big(\sum|\lambda_j|^p\Big)^{1/p}:
	f=\sum \lambda_ja_j \ \text{is an atomic $(p,q,\rho_w,\epsilon)$-representation}\Big\}.
	$$
\end{ddefi}
In the particular case $\epsilon=1$ we write $h^{p,q}_{at,\rho_w}(\Ri^d,w)$ instead of $h^{p,q}_{at,\rho_w,\epsilon}(\Ri^d,w)$.

\begin{ddefi}
	Let $L$ be defined in \eqref{eq-L}. For $p\in (0,1]$, the Hardy space $H^p_L(\Ri^d, w)$ is defined as a completion of the set
	\[
	\Big\{f\in L^2_w(\Ri^d): \mathcal M_L f\in L^p_w(\Ri^d)\Big\}
	\] 
	under the norm
	\[
	\|f\|_{H^p_L(\Ri^d, w)}=\Big\|\mathcal M_L f|\Big\|_{L^p_w(\Ri^d)},
	\]
	where
	\[
	\mathcal M_L f =\sup_{t>0}|e^{-tL}f|.
	\]
\end{ddefi}
The last main result of this paper is the following:
\begin{thrm}
	\label{mainthm 3}
Let $\theta = \min \{\delta,\gamma\}$, where $\delta$ is the constant in \eqref{(M1)} and $\gamma$ is the constant in Proposition \ref{ht prop} (ii).
For all $p\in (\f{n}{n+\theta},1]$ and $q\in [1,\vc]\cap (p,\vc]$  we have
	\[
	h^{p,q}_{at,\rho_w}(\Ri^d,w)\equiv H^p_L(\Ri^d,w).
	\]
\end{thrm}

The organization of the paper is as follows. Section 2 will establish some basic properties for the critical function for the later uses. The proofs of Theorem \ref{main}, Theorem \ref{main 2} and Theorem \ref{mainthm 3} will be given in Section 3, Section 4 and Section 5, respectively.

\bigskip

{\bf Notations.} \quad
Throughout the paper the following set of notation is used without mentioning.
Set $\Ni = \{0, 1, 2, 3, \ldots\}$ and $\Ni^* = \{1, 2, 3, \ldots\}$.
Given a $j \in \Ni$ and a ball $B = B(x,r)$, we let $2^j B = B(x,2^jr)$, $S_0(B) = B$ and $S_j(B) = 2^j B \setminus 2^{j-1} B$ if $j \ge 1$.
For all $a, b \in \Ri$, $a \wedge b = \min\{a,b\}$ and $a \vee b = \max\{a,b\}$.
For all ball $B \subset \Ri^d$ we write $w(B) := \int_B w$.
The constants $C$ and $c$ are always assumed to be positive and independent of the main parameters whose values change from line to line.
For any two functions $f$ and $g$, we write $f \lesssim g$ and $f \sim g$ to mean $f \leq C g$ and $c g \leq f \leq C g$ respectively.
Given a $p \in [1,\infty)$, the conjugate index of $p$ is denoted by $p'$.
If $f$ is defined on $\Ri^d \times \Ri^d$, the gradient with respect to the first variable of $f$ (if exists) is written as $\nabla_1 f$ and this notation is generalized to higher orders in an obvious manner.
We write $L^2(\Ri^d)$ to mean the space of square-integrable function with respect to the Lebesgue measure.
In a weighted setting of Lebesgue spaces we will use the notation $L^2_w(\Ri^d) = L^2(\Ri^d,dw)$.

\section{The critical function $\rho(\cdot,\mu)$ and some properties}

This section will prove some basic properties for the critical function $\rho_w(\cdot,\mu)$. These properties are a conner stone in the proofs of main results. 
\begin{prop} \label{crit}
Let $\rho_w(\cdot,\mu)$ be the function defined in \eqref{critical function}. Then we have the following.
\begin{tabeleq}
\item The function $\rho_w(\cdot,\mu)$ is well-defined, i.e.,  $\rho_w(x,\mu) \in (0,\infty)$ for every $x\in \Ri^d$.
\item For every $x\in \Ri^d$ one has
\[\frac{w(B(x,r))}{r^2} \le \pi(B(x,r)) \le C_1 \, \frac{w(B(x,r))}{r^2},
\] 
with $r = \rho_w(x,\mu)$.

\item If $|x-y| \ls \rho_w(x,\mu)$, then $\rho_w(x,\mu) \sim \rho_w(y,\mu)$.
\item There exist $k_0 > 0$ and $C>1$ such that
\[
C^{-1}m_w(y,\mu) \, \big( 1 + |x-y| \, m_w(y,\mu) \big)^{-k_0/(k_0+1)} 
\le m_w(x,\mu) 
\le Cm_w(y,\mu) \, \big( 1 + |x-y| \, m_w(y,\mu) \big)^{k_0}
\]
for all  $x,y\in \Ri^d$.
\end{tabeleq}
\end{prop}

\begin{proof}
Let $x, y \in \Ri^d$, $r = \rho_w(x,\mu)$ and $R = \rho_w(y,\mu)$.

(i) It follows from \eqref{(M1)} that 
\[
\lim_{t \to 0} \frac{t^2}{w(B(x,t))} \, \pi(B(x,t)) = 0
\ca
\lim_{t \to \infty} \frac{t^2}{w(B(x,t))} \, \pi(B(x,t)) = \infty.
\]
This, in combination with \eqref{(M1)}, implies $\rho_w(x,\mu) \in (0,\infty)$.

(ii) By definition we have
\[
\pi(B(x,r)) = \lim_{t \to r^-} \pi(B(x,t)) \le C_1 \, \frac{w(B(x,r))}{r^2}.
\]
Also
\[
2^{\beta-2} \, C_1 \, \frac{w(B(x,r))}{r^2}
\le C_1 \, \frac{w(B(x,2r))}{4r^2}
\le \pi(B(x,2r))
\le C_1 \, \left( \pi(B(x,r)) + \frac{w(B(x,r))}{r^2} \right),
\]
where we used \eqref{(RD)} in the first step, the definition of $\rho_w$ in the second step and \eqref{(M2)} in the last step.
From this we deduce that
\[
\pi(B(x,r)) \ge \frac{w(B(x,r))}{r^2}.
\]

(iii) Suppose that $|x-y| < Cr$ for some $C > 0$.
Then $B(y,r) \subset B(x,(C+1)r)$.
Using \eqref{(M2)} and (ii) we obtain
\[
\pi(B(x,(C+1)r)) 
\ls \pi(B(x,r)) 
\ls \frac{w(B(x,r))}{r^2}.
\]
Consequently, it follows from \eqref{(M1)}   that
\begin{eqnarray*}
\frac{(tr)^2}{w(B(y,tr))} \, \pi(B(y,tr))
&\le& C_0 \, t^{\delta} \, \frac{r^2}{w(B(y,r))} \, \pi(B(y,r))
\\
&\ls& t^{\delta} \, \frac{r^2}{w(B(x,r))} \, \pi(B(x,(C+1)r))
\\
&\ls& t^\delta < C_1,
\end{eqnarray*}
where $t$ is chosen to be sufficiently small.
Therefore $R \ge tr$ by definition, where we recall that $R=\rho_w(y,\mu)$.
Note that this in turn implies $|x-y| \ls R$.
By swapping the roles of $x$ and $y$ in the above argument, we then obtain $R \ls r$.

(iv) The case $|x-y| < R$ follows from (iii).
So we assume that $|x-y| \ge R$.
Let $j \in \Ni^*$ be such that $2^{j-1} \, R \le |x-y| \le 2^j R$.
Then $B(x,R) \subset B(y,(2^j+1)R)$.
It follows from (ii) and \eqref{(M2)} that
\begin{eqnarray*}
\pi(B(x,R)) 
&\le& (C_1 + 2^{2d-2})^{j} \, \frac{w(B(y,R))}{R^2}
\ls (C_1 + 2^{2d-2})^{j} \, (1 + 2^j)^d \frac{w(B(x,R))}{R^2}
\\
&\ls& (1 + 2^j)^b \, \frac{w(B(x,R))}{R^2},
\end{eqnarray*}
where $b := d + \log_2(C_1+2^{2d-2})$.
Using \eqref{(M1)},
\begin{eqnarray*}
\frac{(tR)^2}{w(B(x,tR))} \, \pi(B(x,tR))
&\le& C_0 \, t^{\delta} \, \frac{R^2}{w(B(x,R))} \, \pi(B(x,R))
\ls t^\delta \, (1 + 2^j)^b < C_1,
\end{eqnarray*}
where we choose $t = \left(\frac{C_1}{2} \, (1 + 2^j)^{-b}\right)^{1/\delta}$.
So the definition of $\rho_w$ gives $r \ge tR$ or equivalently
\begin{equation} \label{ineq 1}
m_w(x,\mu) 
\le \frac{ m_w(y,\mu)}{t} 
\sim m_w(y,\mu) \, \big( 1 + |x-y| \, m_w(y,\mu) \big)^{k_0},
\end{equation}
where $k_0 := b/\delta$.

For the remaining inequality, using \eqref{ineq 1} we obtain that
\[
1 + |x-y| \, m_w(x,\mu) \ls \big( 1 + |x-y| \, m_w(y,\mu) \big)^{k_0+1}.
\]
With this in mind we apply \eqref{ineq 1} again to obtain
\[
m_w(y,\mu) \gtrsim m_w(x,\mu) \, \big( 1 + |x-y| \, m_w(x,\mu) \big)^{-k_0/(k_0+1)}.
\]

The proof is complete.
\end{proof}

\begin{lemm} \label{big R}
	There exist constants $N_0 \in \Ni$ and $C > 0$ such that
	\[
	\frac{R^2}{w(B(x,R))} \, \pi(B(x,R)) \le C \, \Big(\f{R}{\rho_w(x,\mu)}\Big)^{N_0}
	\]
	for all $x \in \Ri^d$ and $R > 0$ such that $R \ge \rho_w(x,\mu)$.
\end{lemm}

\begin{proof}
	Let $x \in \Ri^d$ and $r_0 = \rho_w(x,\mu)$.
	By the definition of $\rho_w(x,\mu)$,
	\[
	\pi(B(x,r)) \ge C_1 \, \frac{w(B(x,r))}{r^2}
	\]
	for all $r \ge r_0$.
	
	Therefore, \eqref{(M2)} implies that $\pi$ is a doubling measure on all balls $B(x,r)$ with $r \ge r_0$, i.e.,
	\[
	\pi(B(x,2r))\lesssim \pi(B(x,r))
	\]
	for all $r\ge r_0=\rho_w(x,\mu)$.
	
	Let $j_0 \in \Ni$ be such that $2^{j_0} r_0 \le R < 2^{j_0+1} r_0$.
	Then
	\[
	\pi(B(x,R)) \le C_1^{j_0+1} \, \pi(B(x,r_0)) \le C^{j_0} \, \frac{w(B(x,r_0))}{r_0^2},
	\]
	where we used Proposition \ref{crit} (ii) in the last step.
	
	Hence,
	\[
	\frac{R^2}{w(B(x,R))} \, \pi(B(x,R))
	\le C_1^{j_0} \, \frac{R^2}{r_0^2} \, \frac{w(B(x,r_0))}{w(B(x,R))}
	\ls \left( \frac{R}{r_0} \right)^{\log_2 C_1}.
	\]
	This justifies our claim.
\end{proof}

\begin{lemm} \label{Schwartz}
	Let $x \in \Ri^d$ and $t > 0$. 
	If $\sqrt{t} \le a \, \rho_w(x,\mu)$ for some $a>0$, then  there exists a $C=C(a) > 0$ such that
	\begin{equation}\label{eq1-lem Schwartz}
	\int_{\Ri^d} \frac{1}{w(B(x\wedge y,\sqrt{t}))} \, \exp\Big(-\f{|x-y|^2}{ct}\Big) \, d\pi(y)
	\le \frac{C}{t} \cdot \left(\frac{\sqrt{t}}{\rho_w(x,\mu)}\right)^\delta,
	\end{equation}
	where $w(B(x\wedge y,\sqrt{t}))=\min\{w(B(x,\sqrt{t})),w(B(y,\sqrt{t}))\}$ and $\delta$ is the constant in \eqref{(M1)}.
	
	If $\sqrt{t} \ge a \, \rho_w(x,\mu)$ for some $a>0$, then  there exists a $C=C(a) > 0$ such that
	\begin{equation}\label{eq2-lem Schwartz}
	\int_{\Ri^d} \frac{1}{w(B(x\wedge y,\sqrt{t}))} \, \exp\Big(-\f{|x-y|^2}{ct}\Big) \, d\pi(y)
	\le \frac{C}{t} \cdot \left(\frac{\sqrt{t}}{\rho_w(x,\mu)}\right)^{N_0},
	\end{equation}
	where $N_0$ is the constant in Lemma \ref{big R}.
\end{lemm}

\begin{proof}
	We first prove \eqref{eq1-lem Schwartz}. Setting $r=\rho_w(x,\mu)$, $B=B(x,\sqrt t)$ and taking $j_0\in \mathbb{N}$ such that $2^{j_0}\sqrt t\le r< 2^{j_0+1}\sqrt t$, then we have
	\begin{equation}\label{eq-sum Ej big small}
	\begin{aligned}
	\int_{\Ri^d} &\frac{1}{w(B(x\wedge y,\sqrt{t}))} \, \exp\Big(-\f{|x-y|^2}{ct}\Big) \, d\pi(y)\\
	&= \sum_{j=0}^\infty \int_{S_j(B)}\frac{1}{w(B(x\wedge y,\sqrt{t}))} \, \exp\Big(-\f{|x-y|^2}{ct}\Big) \, d\pi(y)\\
	&=:\sum_{j=0}^\infty E_j=\sum_{j=0}^{j_0} E_j +\sum_{j=j_0+1}^\infty E_j.
	\end{aligned}
	\end{equation}
	For $j=0,1,\ldots, j_0$ we have
	\begin{equation}
	\label{eq1-proof 7.3}
	\begin{aligned}
	E_j&\lesi e^{-c2^{2j}} \frac{\pi(2^jB)}{w(B(x\wedge y,\sqrt{t}))}=e^{-c2^{2j}} \frac{\pi(2^jB)}{w(2^jB)}\f{w(2^jB)}{w(B(x\wedge y,\sqrt{t}))}\\
	&\lesi 2^{2jd}e^{-c2^{2j}}\frac{\pi(2^jB)}{w(2^jB)}\\
	&\lesi e^{-c2^{2j}}\frac{\pi(2^jB)}{w(2^jB)},
	\end{aligned}
	\end{equation}
	where in the second inequality we used \eqref{doubling from w}.
	
	Note that for $j=0,1,\ldots, j_0$ one has $2^{j}\sqrt t\le r=\rho_w(x,\mu)$. Hence, owing to \eqref{(M1)} and Proposition \ref{crit}, we have
	\[
	\frac{\pi(2^jB)}{w(2^jB)}\lesi \Big(\f{2^j\sqrt t}{r}\Big)^\delta \f{1}{2^{2j}t}\f{r^2 v(B(x,r))}{w(B(x,r))}\sim \Big(\f{2^j\sqrt t}{r}\Big)^\delta \f{1}{2^{2j}t}.
	\]
	Plugging this into the estimate of $E_j$, we can simplify that
	\[
	E_j\lesi e^{-c2^{2j}}\f{1}{t}\Big(\f{\sqrt t}{r}\Big)^\delta=e^{-c2^{2j}}\f{1}{t}\Big(\f{\sqrt t}{\rho_w(x,\mu)}\Big)^\delta
	\]
	for all $j\le j_0$, which implies
	\begin{equation}
	\label{eq- Ej small}
	\sum_{j=0}^{j_0} E_j\lesi \f{1}{t}\Big(\f{\sqrt t}{\rho_w(x,\mu)}\Big)^\delta.
	\end{equation}
	
	For all $j> j_0$, similarly to \eqref{eq1-proof 7.3} we have
	\[
	\begin{aligned}
	E_j&\lesi e^{-c2^{2j}}\frac{\pi(2^jB)}{w(2^jB)}
	\end{aligned}
	\ca
	2^j\sqrt t\ge r=\rho_w(x,\rho).
	\]
	Applying Lemma \ref{big R} and the fact that $\sqrt t\le Cr$,
	\[
	\begin{aligned}
	E_j&\lesi e^{-c2^{2j}}\f{1}{2^{2j}t}\Big(\f{2^j\sqrt t}{r}\Big)^{N_0}\lesi  e^{-c2^{2j}} 2^{jN_0}\f{1}{t}\lesi e^{-c2^{2j}}\f{1}{t},
	\end{aligned}
	\] 
	which implies
	\[
	\begin{aligned}
	\sum_{j> j_0}E_j&\lesi e^{-c2^{2j_0}}\f{1}{t}\lesi 2^{-j_0 \delta}\f{1}{t}.
	\end{aligned}
	\]
	This, along with the fact that $2^{-j_0}\sim \sqrt t/r$, yields that 
	\[
	\sum_{j> j_0}E_j\lesi \frac{C'}{t} \cdot \left(\frac{\sqrt{t}}{\rho_w(x,\mu)}\right)^\delta.
	\]
	Collecting this estimate and the estimates \eqref{eq- Ej small} and \eqref{eq-sum Ej big small} we deduce to the desired estimate \eqref{eq1-lem Schwartz}.
	
	\bigskip
	
	The proof of \eqref{eq2-lem Schwartz} can be done similarly. Hence, we omit the details.
	
	This completes our proof.
\end{proof}

We end this section with the following useful lemma regarding a covering result of a family of balls whose radii are equal to the values of the critical function at their centers. 
\begin{lemm} \label{xalpha psialpha}
	There exist a sequence $(x_j)_{j \in \Ni} \subset \Ri^d$ and a family of functions $(\psi_j)_{j \in \Ni}$ such that the following hold.
	\begin{tabeleq}
		\item $\bigcup_{j \in \Ni} B_j = \Ri^d$, where $\rho_j = \rho_w(x_j,\mu)$ and $B_j=B(x_j, \rho_j)$ for all $j \in \Ni$.
		\item For all $\tau \geq 1$ there exist constants $C, \zeta_0 > 0$ such that 
		\[
		\sum_{j \in \Ni} \chi_{B(x_j, \tau \rho_j)} \leq C \, \tau^{\zeta_0}.
		\]
		\item $\supp \psi_j \subset B(x_j, \rho_j)$ and $0 \leq \psi_j \leq 1$.
		\item  $|\nabla\psi_j(x)| \ls 1/\rho_j$ for all $x, y \in \Ri^d$.
		\item  $\sum_{j \in \Ni} \psi_j = 1$.
	\end{tabeleq}
\end{lemm}

\begin{proof}
	We note that $\rho_w(\cdot,\mu)$ acquires all the properties analogous to those of $\rho(\cdot,\mu)$ given in \cite{Shen}.
	Hence the proof for this lemma is done verbatim as in \cite[Proof of Lemma 3.3]{Shen}.
\end{proof}
\section{Upper bounds for the fundamental solution $\Gamma_\mu(x,y)$}

This section is devoted to the proof of Theorem \ref{main}. To do this, we first establish some solution/subsolution estimates for the equation $(L+i\tau)u=f$. Before coming to the details we need to set up the formal definition of the operator $L$.

In what follows we denote
\[
W^{1,2}_{w,\loc}(\Ri^d)
:= \big\{ u \in L^2_{w,\loc}(\Ri^d): \D_j u \in L^2_{w,\loc}(\Ri^d) \mbox{ for all } j \in \{1,\ldots,d\} \big\}.
\]

\subsection{The formal definition of $L$}\label{operator def}
We first recall the Poincar\'e's inequality in  \cite[Lemma 5]{KS}.

\begin{prop}\label{Poincare}
	Let $x_0 \in \Ri^d$, $R > 0$ and $B = B(x_0,R)$.
	Then 
	\[
	\int_B \int_B |\phi(x) - \phi(y)|^2 \, dw(x) \, dw(y)
	\le C \, R^2 \, w(B) \, \int_B |\nabla \phi(x)|^2 \, dw(x)
	\]
	for all $\phi \in C^1(\overline{B})$.
\end{prop}

We now state the following result which plays a key role in the construction of the formal definition of the operator $L$.
\begin{prop} \label{norm equiv}
Let $u \in W^{1,2}_{w,\loc}(\Ri^d)$ such that $\nabla u \in L^2_w(\Ri^d)$.
Then the following hold.
\begin{tabel}
\item If $u \in L^2_\pi(\Ri^d)$ then $m_w(\cdot,\mu) \, u \in L^2_w(\Ri^d)$ and 
\[
\int_{\Ri^d} |u|^2 \, m_w(\cdot,\mu)^2 \, dw
\ls \int_{\Ri^d} |\nabla u|^2 \, dw + \int_{\Ri^d} |u|^2 \, d\pi.
\]
\item If $m_w(\cdot,\mu) \, u \in L^2_w(\Ri^d)$ then $u \in L^2_\pi(\Ri^d)$ and
\[
\int_{\Ri^d} |u|^2 \, d\pi
\ls \int_{\Ri^d} |\nabla u|^2 \, dw + \int_{\Ri^d} |u|^2 \, m_w(\cdot,\mu)^2 \, dw.
\]
\end{tabel}
\end{prop}

\begin{proof}
We prove (a) only.
The proof for (b) is done analogously.

Let $x_0 \in \Ri^d$ and $r_0 = \rho_w(x_0,\mu)$.
Set $B = B(x_0,r_0)$.
By Proposition \ref{crit} (ii) we have 
\begin{eqnarray*}
I 
&:=& \int_B \left( \frac{w(B)}{r_0^2} \wedge \pi(B) \right) \, |u|^2 \, dw
\ge \frac{w(B)}{r_0^2} \, \int_B |u|^2 \, dw.
\end{eqnarray*}

Also it follows from Proposition \ref{Poincare} that
\begin{eqnarray*}
I 
&\ls& \int_B \int_B \frac{1}{r_0^2} \, |u(x)-u(y)|^2 \, dw(x) \, dw(y)
+ w(B) \, \int_B |u(y)|^2 \, d\pi(y)
\\
&\ls& w(B) \, \left( \int_B |\nabla u(x)|^2 \, dw(x) + \int_B |u(x)|^2 \, d\pi(x) \right).
\end{eqnarray*}
Hence
\[
\frac{1}{r_0^2} \, \int_B |u|^2 \, dw
\ls \int_B |\nabla u|^2 \, dw + \int_B |u|^2 \, d\pi,
\]
or equivalently
\[
\int_B |u|^2 \, m_w(\cdot,\mu)^2 \, dw
\ls \int_B |\nabla u|^2 \, dw + \int_B |u|^2 \, d\pi,
\]
as $m_w(x,\mu) \sim 1/r_0$ for all $x \in B$ by Proposition \ref{crit}(iii).

Hence, let $\{B_j\}_{j\in \mathbb N}$ be the family of balls in Lemma \ref{xalpha psialpha}. Then we have
\[
\int_{B_j} |u|^2 \, m_w(\cdot,\mu)^2 \, dw
\ls \int_{B_j} |\nabla u|^2 \, dw + \int_{B_j} |u|^2 \, d\pi
\]
for each $j\in \mathbb{N}$.

Summing over all $j\in \mathbb{N}$ and using (i) and (ii) of Lemma \ref{xalpha psialpha}, we arrive at the conclusion.
\end{proof}

The following result is a direct consequence of Proposition \ref{norm equiv}.
\begin{coro} \label{H equiv}
Let 
\begin{equation} \label{H}
H := \big\{ u \in W^{1,2}_{w,\loc}(\Ri^d): \nabla u \in L^2_w(\Ri^d) \mbox{ and } m_w(\cdot,\mu) \, u \in L^2_w(\Ri^d) \big\}
\end{equation}
be equipped with the norm
\[
\|u\|_H = \int_{\Ri^d} |\nabla u|^2 \, dw + \int_{\Ri^d} m_w(\cdot,\mu)^2 \, |u|^2 \, dw,
\]
and
\[
H' := \big\{ u \in W^{1,2}_{w,\loc}(\Ri^d): \nabla u \in L^2_w(\Ri^d) \mbox{ and } u \in L^2_w(\Ri^d,d\mu) \big\}
\]
be equipped with the norm
\[
\|u\|_{H'} = \int_{\Ri^d} |\nabla u|^2 \, dw + \int_{\Ri^d} |u|^2 \, w \, d\mu.
\]
Then $H=H'$ with equivalent norms. 

Moreover, $H$ is a Hilbert space (with respect to the induced inner product).
\end{coro}

Consider the quadratic form
\[
\gota(u,v) = \int_{\Ri^d} A \, \nabla u \cdot \nabla v \, dx + \int_{\Ri^d} u \, v \, d\pi
\]
on the domain
\[
D(\gota) 
= H \cap L^2_w(\Ri^d)
=  \big\{ u \in W^{1,2}_w(\Ri^d): u \in L^2_\pi(\Ri^d) \big\},
\]
where $d\pi = w \, d\mu$ and $H$ is given by \eqref{H} and
\[
W^{1,2}_w(\Ri^d) 
:= \big\{ u \in L^2_w(\Ri^d): \D_j u \in L^2_w(\Ri^d) \mbox{ for all } j \in \{1,\ldots,d\} \big\}.
\]
We endow $D(\gota)$ with the graph norm
\[
\|u\|_{D(\gota)} = \gota(u,u) + \|u\|_{L^2_w(\Ri^d)}
\]
for all $u \in D(\gota)$.
It follows from \eqref{degenerate} and Corollary \ref{H equiv} that 
\begin{equation} \label{graph norm equiv}
\|u\|_{D(\gota)}
\sim \int_{\Ri^d} |\nabla u|^2 \, dw + \int_{\Ri^d} |u|^2 \, m_w(\cdot,\mu)^2 \, dw + \int_{\Ri^d} |u|^2 \, dw
\end{equation}
for all $u \in D(\gota)$.

It is easy to see that $\gota$ is positive and symmetric.
We will show in addition that $\gota$ is also densely defined and closed.

We need the following auxiliary result.
In what follows define
\[
W^{1,2}_{w,0}(B) := \overline{\big( C_c^\infty(B), \|\cdot\|_{W^{1,2}_w} \big)}.
\]

\begin{lemm} \label{W12 to L2mu}
Let $B \subset \Ri^d$ be a ball.
Then the embedding $W^{1,2}_{w,0}(B) \hookrightarrow L^2(B,d\pi)$ is continuous.
\end{lemm}

\begin{proof}
Let $\{x_j\}_{j \in \Ni}$ and $\{\psi_j\}_{j \in \Ni}$ be as in Lemma \ref{xalpha psialpha}.
Since $B$ is compact we can cover it by a finite number of balls $B_j := B(x_j,\rho_j)$.
Without loss of generality assume that $B \subset \cup_{j=1}^{j_0} B_j$ for some $j_0 \in \Ni^*$.

Therefore using Proposition \ref{norm equiv} one has
\begin{eqnarray*}
\int_B |u|^2 \, d\pi
&\ls& \int_B |\nabla u|^2 \, dw + \int_B |u|^2 \, m_w(\cdot,\mu)^2 \, dw
\\
&\le& \int_B |\nabla u|^2 \, dw + \sum_{j=1}^{j_0} \int_{B \cap B_j} |u|^2 \, m_w(\cdot,\mu)^2 \, dw
\\
&\ls& \int_B |\nabla u|^2 \, dw + \sum_{j=1}^{j_0} m_w(x_j,\mu)^2 \int_{B \cap B_j} |u|^2 \, dw 
\\
&\le& \left(1 \vee \sum_{j=1}^{j_0} m_w(x_j,\mu)^2 \right) \, \|u\|_{W^{1,2}_w(B)}
< \infty
\end{eqnarray*}
for all $u \in W^{1,2}_{w,0}(B)$, where we used Proposition \ref{crit}(iii) in the third step.

This verifies our claim.
\end{proof}

\begin{lemm}\label{lemm1}
The space $C_c^\infty(\Ri^d)$ is a form core for $\gota$.
Consequently, $\gota$ is densely defined in $L^2_w(\Ri^d)$.
\end{lemm}

\begin{proof}
Let $f \in D(\gota)$.
By multiplying $f$ with a cut-off function when necessary, we may assume that $\supp f \subset B$ and $f \in W^{1,2}_{w,0}(B)$ for some ball $B \subset \Ri^d$.
So there exists a sequence $\{f_j\}_{j \in \Ni} \subset C_c^\infty(B)$ such that $\lim_{j \to \infty} f_j = f$ in $W^{1,2}_{w,0}(B)$.
By Lemma \ref{W12 to L2mu} we also have that $\lim_{j \to \infty} f_j = f$ in $L^2(B,d\pi)$.
Hence $\lim_{j \to \infty} f_j = f$ in $D(\gota)$.
To finish note that $C_c^\infty(\Ri^d)$ is dense in $L^2_w(\Ri^d)$ by \cite[Theorem 1.1]{NTY}.
\end{proof}

\begin{lemm}\label{lemm2}
The form $\gota$ is closed in $L^2_w(\Ri^d)$.
\end{lemm}

\begin{proof}
Let $\{f_j\}_{j \in \Ni} \subset D(\gota)$ be a Cauchy sequence.
Then $\{f_j\}_{j \in \Ni}$ is Cauchy in $W^{1,2}_w(\Ri^d)$ and $L^2_\pi(\Ri^d)$.
Hence there exist functions $u \in W^{1,2}_w(\Ri^d)$ and $f \in L^2_\pi(\Ri^d)$ such that $\lim_{j \to \infty} f_j = u$ in $W^{1,2}_w(\Ri^d)$ and $\lim_{j \to \infty} f_j = f$ in $L^2_\pi(\Ri^d)$.
By using a subsequence if necessary we may conclude that $\lim_{j \to \infty} f_j = u$ a.e.\ in $\Ri^d$.
Hence $u = f$.
It follows that $u \in D(\gota)$ and $\lim_{j \to \infty} f_j = u$ in $D(\gota)$.
\end{proof}

We are ready to give the formal definition of the operator $L$. From Lemmas \ref{lemm1} and \ref{lemm2}, there exists a unique self-adjoint operator 
\[
Lu := -\frac{1}{w} \, \di(A \, \nabla u) +\mu \, u
\]
on the domain
\[
D(L) = \{u \in D(\gota): Lu \in L^2_w(\Ri^d) \}
\]
such that
\[
\gota(u,v) = \langle Lu, v \rangle_{L^2_w(\Ri^d)} 
\]
for all $u \in D(L)$ and $v \in D(\gota)$.

\subsection{Some estimates on solutions to the equation $L_0u=f$}

Define
\[
L_0 = -\frac{1}{w} \, \di( A \, \nabla).
\]
Let $\Gamma_0(\cdot,\cdot)$ be its fundamental solution in $\Ri^d$.

\begin{ddefi} \label{weak sol of L0}
	Let $\Omega \subset \Ri^d$ be open.
	Let $u \in W^{1,2}_{w,\loc}(\Omega)$ and $f \in L^1_{w,\loc}(\Omega)$.
	Then $u$ is called a weak solution of $L_0 u = f$ in $\Omega$ if
	\[
	\int_\Omega A \, \nabla u \cdot \nabla \psi \, dx = \int_\Omega f \, \psi \, dw
	\]
	for all $\psi \in C^1_c(\Omega)$, where $dw = w \, dx$.
\end{ddefi}

\begin{ddefi} \label{sub-sol L0}
	Let $\Omega \subset \Ri^d$ be open and $u \in W^{1,2}_{w,\loc}(\Omega)$.
	Then $u$ is called a sub-solution of $L_0$ in $\Omega$ if
	\[
	\int_\Omega A \, \nabla u \cdot \nabla \psi \, dx
	\le 0
	\]
	for all non-negative function $\psi \in C^1_c(\Omega)$.
\end{ddefi}

The following two estimates are taken from \cite[Lemmas 8 and 7]{KS} respectively.
\begin{lemm}\label{u sub L0}
	Let $x \in \Ri^d$, $R > 0$.
	Let $u$ be a non-negative sub-solution of $L_0$ in $B(x,2R)$.
	Then for all $\sigma \in (0,1)$ there exists a constant $C = C(\sigma)$ such that
	\[
	\sup_{B(x,\sigma R)} u \le C \, \frac{1}{w(B(x,R))} \, \int_{B(x,R)} u \, dw.
	\]
\end{lemm}

\begin{prop}\label{funda0}
	There exists a $C > 0$ such that
	\[
	0 \le \Gamma_0(x,y) \le C \, \frac{|x-y|^2}{w(B(x,|x-y|))}
	\]
	for all $x, y \in \Ri^d$.
\end{prop}

\subsection{Existence of solutions/subsolutions}

\begin{ddefi} \label{weak sol of L}
Let $\Omega \subset \Ri^d$ be open.
Let $u \in W^{1,2}_{w,\loc}(\Omega)$ and $f \in L^1_{w,\loc}(\Omega)$.
Then $u$ is called a weak solution of $Lu = f$ in $\Omega$ if
\[
\int_\Omega A \, \nabla u \cdot \nabla \psi \, dx + \int_\Omega u \, \psi \, d\pi
= \int_\Omega f \, \psi \, dw
\]
for all $\psi \in C^1_c(\Omega)$, where we remind that $d\pi = w \, d\mu$ and $dw = w \, dx$.
\end{ddefi}

\begin{ddefi}
Let $\Omega \subset \Ri^d$ be open and $u \in W^{1,2}_{w,\loc}(\Omega)$.
Then $u$ is called a sub-solution of $L$ in $\Omega$ if
\[
\int_\Omega A \, \nabla u \cdot \nabla \psi \, dx + \int_\Omega u \, \psi \, d\pi
\le 0
\]
for all non-negative function $\psi \in C^1_c(\Omega)$.
\end{ddefi}

\begin{prop} \label{unique}
Let $f \in L^1_{w,\loc}(\Ri^d)$ be such that $m_w(\cdot,\mu)^{-1} \, f(\cdot) \in L^2_w(\Ri^d)$.
Then $Lu = f$ has a unique weak solution $u_f \in H$, where $H$ is defined by \eqref{H}.
\end{prop}

\begin{proof}
This is immediate from the Lax-Milgram theorem.
\end{proof}

Schwartz kernel theorem now ensures that there exists a unique distributional $\Gamma_\mu(\cdot,\cdot)$ such that the representation
\[
u_f(x) = \int_{\Ri^d} \Gamma_\mu(x,y) \, f(y) \, dw(y)
\]
holds for a.e.\ $x \in \Ri^d$, $f \in L^2_{w,c}(\Ri^d)$, where $u_f$ is as in Proposition \ref{unique}.
Such a $\Gamma_\mu(\cdot,\cdot)$ in fact enjoys further properties as stated in Proposition \ref{funda prop} below.

Recall that 
\[
L_0 = -\frac{1}{w} \, \di( A \, \nabla).
\]

\begin{lemm} \label{sub lem 1}
Let $u \in W^{1,2}_{w,\loc}(\Omega)$ be a sub-solution of $Lu = 0$ in $\Omega$.
Then $u^+ := u \vee 0$ is a sub-solution of $L_0$ in $\Omega$. 
\end{lemm}

\begin{proof}
Let $\phi \in C_c^1(\Omega)$ be positive and set 
\[
\psi = \phi \times \frac{u^+}{u^+ + \epsilon} \in C_c^1(\Omega)
\]
for each $\epsilon > 0$.
Since $u$ is a sub-solution of $Lu = f$ in $\Omega$, it follows that
\begin{equation} \label{sub 1}
\int_\Omega A \, \nabla u \cdot \nabla \psi \, dx + \int_\Omega u \, \psi \, d\pi \le 0.
\end{equation}
Observe that the left-hand side equals to
\begin{eqnarray*}
&& \int_\Omega A \, \nabla u \cdot (\nabla \phi) \, \frac{u^+}{u^+ + \epsilon} \, dx
+ \int_\Omega A \, \nabla u \cdot (\nabla u^+)  \, \frac{\epsilon \, \phi}{(u^+ + \epsilon)^2} \, dx
+ \int_\Omega u \, \phi \, \frac{u^+}{u^+ + \epsilon} \, d\pi
\\
&=& \int_{[u>0]} A \, \nabla u \cdot (\nabla \phi) \, \frac{u}{u + \epsilon} \, dx
+ \int_{[u>0]} A \, \nabla u \cdot (\nabla u) \, \frac{\epsilon \, \phi}{(u + \epsilon)^2} \, dx
+ \int_{[u>0]} \phi \, \frac{u^2}{u + \epsilon} \, d\pi
\end{eqnarray*}
for all $\epsilon > 0$, where we used 
\[
\nabla (u^+) = 
\left\{
\begin{array}{ll}
\nabla u & \mbox{on } [u > 0],
\\
0 & \mbox{otherwise},
\end{array}
\right.
\]
(cf.\ \cite[Lemma 7.6]{GT}) and the definition of $u^+$.

By taking the limits on both sides of \eqref{sub 1} as $\epsilon \longrightarrow 0$ we obtain
\begin{eqnarray*}
\int_{\Omega} A \, \nabla u^+ \cdot \nabla \phi \, dx
&=& \int_{[u>0]} A \, \nabla u \cdot \nabla \phi \, dx
\le \int_{[u>0]} A \, \nabla u \cdot \nabla \phi \, dx + \int_{[u>0]} \phi \, u \, d\pi 
\le 0.
\end{eqnarray*}
Since $0 \le \phi \in C_c^1(\Omega)$ is arbitrary, a density argument justifies the claim.
\end{proof}

\begin{lemm} \label{sub lem 2}
Let $u \in W^{1,2}_{w,\loc}(\Omega)$ be a weak solution of $Lu = 0$ in $\Omega$.
Then $u^+$ and $|u|$ are sub-solutions of $L_0$ in $\Omega$. 
\end{lemm}

\begin{proof}
By hypothesis, $u$ and $-u$ are sub-solutions of $Lu = f$ in $\Omega$.
An application of Lemma \ref{sub lem 1} yields that $u^+$ and $u^- := (-u)^+$ are sub-solutions of $L_0$ in $\Omega$.
Hence $|u| = |u^+| + |u^-|$ is also a sub-solution of $L_0$ in $\Omega$.
\end{proof}

\begin{lemm} \label{sub lem 3}
Let $u \in W^{1,2}_{w,\loc}(\Ri^d)$ and $f \in L^1_{w,\loc}(\Ri^d)$.
Suppose that $f \ge 0$, $u$ is a weak solution of $Lu = f$ in $\Ri^d$, and
\begin{equation} \label{con 1}
\lim_{R \to \infty} \sup_{|x|=R} \frac{1}{w(B(x,R/2))} \, \int_{B(x,R/2)} |u| \, dw = 0.
\end{equation}
Then $u \ge 0$ in $\Ri^d$.
\end{lemm}

\begin{proof}
Since $f \ge 0$ we deduce that $-u$ is a sub-solution in $\Ri^d$.
It follows from Lemma \ref{sub lem 1} that $u^-$ is a sub-solution of $L_0$ in $\Ri^d$.
The maximum principle in \cite[Theorem 2.3.8]{FKS} now implies
\[
\sup_{B(0,R)} u^- \le \sup_{\D B(0,R)} u^-.
\]
However, according to Lemma \ref{u sub L0}, for all $x \in \D B(0,R)$ we have
\begin{eqnarray*}
u^-(x) 
&\le& \frac{1}{w(B(x,R/2))} \, \int_{B(x,R/2)} u^- \, dw
\le \frac{1}{w(B(x,R/2))} \, \int_{B(x,R/2)} |u| \, dw 
\longrightarrow 0
\end{eqnarray*}
as $R \longrightarrow \infty$ by hypothesis.
Hence $u^- = 0$ in $\Ri^d$.
This in turn implies $u = u^+ \ge 0$.
\end{proof}

\begin{lemm} \label{B01}
Let $R > 0$.
Then there exists a $C > 0$ such that 
\[
0 < R^\beta \, w(B(0,1)) \le C \, w(B(x,R/2)) 
\]
for all $x \in \Ri^d$ such that $R \le |x| \le 2R$.
\end{lemm}

\begin{proof}
Since $B(x,R/2) \subset B(0,3R)$, we also have 
\[
w(B(x,R/2))
\ge \left( \frac{|B(x,R/2)|}{|B(0,3R)|} \right)^{2d} \, w(B(0,3R))
\gtrsim R^\beta \, w(B(0,1)) 
> 0,
\]
where we used \eqref{doubling from w} and \eqref{(RD)} in the first and second steps, respectively.
\end{proof}

\begin{prop} \label{sub lem 4}
Let $f \in L^2_{w,c}(\Ri^d)$ be positive.
Let $u = u_f$, where $u_f$ is given by Proposition \ref{unique}.
Then 
\[
0 \le u(x) \le \int_{\Ri^d} \Gamma_0(x,y) \, f(y) \, dw(y),
\]
where $\Gamma_0$ is the fundamental solution of $L_0$ in $\Ri^d$.
\end{prop}

\begin{proof}
Let $R > 0$.
It follows from Proposition \ref{crit}(iv) that
\begin{equation} \label{check 1}
\frac{1}{R^2} \, \int_{R/2 \le |x| \le 2R} |u|^2 \, dw
\ls \int_{R/2 \le |x| \le 2R} m_w(\cdot,\mu)^2 \, |u|^2 \, dw
\ls \|u\|_H^2,
\end{equation}
where $H$ is given by \eqref{H}.

This leads to
\begin{eqnarray*}
\lim_{R \to \infty} \sup_{|x|=R} \frac{1}{w(B(x,R/2))} \, \int_{B(x,R/2)} |u| \, dw
&\le& \lim_{R \to \infty} \sup_{|x|=R} \left( \frac{1}{w(B(x,R/2))} \, \int_{B(x,R/2)} |u|^2 \, dw \right)^{1/2}
\\
&=& \lim_{R \to \infty} \sup_{|x|=R} \left( \frac{R^2}{w(B(x,R/2))} \, \frac{1}{R^2} \, \int_{B(x,R/2)} |u| \, dw  \right)^{1/2}
\\
&\ls& \lim_{R \to \infty} \sup_{|x|=R} \left( \frac{1}{w(B(0,1))} \, \frac{1}{R^{\beta - 2}} \, \frac{1}{R^2} \, \int_{B(x,R/2)} |u| \, dw  \right)^{1/2}
\\
&=& 0,
\end{eqnarray*}
where we used Lemma \ref{B01} and the fact that $\beta > 2$ in the second-to-last step.
This verifies \eqref{con 1}.
Therefore $u \ge 0$ by Lemma \ref{sub lem 3}.

For the remaining inequality set 
\[
v(x) = \int_{\Ri^d} \Gamma_0(x,y) \, f(y) \, dw(y).
\]
Then $v \in W^{1,2}_{w,\loc}(\Ri^d)$, $L_0 v = f$ in $\Ri^d$ and $v \ge 0$ (cf.\ \cite[Theorem 1.3]{CW}).
Also $u-v$ is a sub-solution of $Lu = 0$ in $\Ri^d$.
Lemma \ref{sub lem 1} now implies that $(u-v)^+$ is a sub-solution of $L_0$ in $\Ri^d$.
Next we use the maximal principle in \cite[Theorem 2.3.8]{FKS} and Lemma \ref{sub lem 2} to derive
\[
\sup_{B(0,R)} (u-v)^+ 
\le \sup_{\D B(0,R)} (u-v)^+ 
\ls \frac{1}{w(B(0,R))} \, \int_{R/2 \le |x| \le 2R} \big( |u(x)| + |v(x)| \big) \, dw(x).
\]
Note that Proposition \ref{funda0} implies $v(x) = O(\frac{R^2}{w(B(x,R))})$ for all $x \in \Ri^d$ such that $R/2 \le |x| \le 2R$.
This together with \eqref{check 1} gives $(u-v)^+ = 0$ in $\Ri^d$.
Hence $u \le v$ in $\Ri^d$.
\end{proof}

We end this subsection with a domination property of the fundamental solutions of $L+i\tau$ for $\tau \in \mathbb{R}$.

\begin{prop} \label{Gam poitwise bound}
Let $\tau\in \mathbb{R}$ and $\Gamma_\mu(\cdot,\cdot,\tau)$ be the fundamental solution of $L + i\tau$. Then there exists a $C = C(d) > 0$ such that 
\[
|\Gamma_\mu(x,y,\tau)| \le \Gamma_\mu(x,y) \le C \, \Gamma_0(x,y)
\]
for all $x, y \in \Ri^d$.
\end{prop}

\begin{proof}
Let $\tau \in \Ri$ and $x, y \in \Ri^d$.
Using Lemma \ref{sub lem 4} and Lebesgue differentiation theorem (see for example  \cite[Theorem]{Tol}) we easily obtain
\[
\Gamma_\mu(x,y) \le C \, \Gamma_0(x,y).
\]
For the first inequality, using functional calculus one has
\[
\Gamma_\mu(x,y,\tau)
= \int_0^\infty e^{-it\tau} \, k_t(x,y) \, dt,
\]
where we recall that $k_t(x,y)$ is the heat kernel of $L$.

Since $k_t(x,y)\ge 0$ (see \eqref{k<h} below), we have
\[
|\Gamma_\mu(x,y,\tau)| 
\le \int_0^\infty k_t(x,y) \, dt
= \Gamma_\mu(x,y).
\]
The claim now follows.
\end{proof}

\begin{prop} \label{funda prop}
	The following statements hold.
	\begin{tabeleq}
		\item $\Gamma_\mu(x,\cdot) \in L^p_\loc(\Ri^d)$ for all $p \in (0, \frac{d}{d-2})$ and $x \in \Ri^d$.
		
		\item For all $f \in L^2_{w,c}(\Ri^d)$ the function
		\[
		u(\cdot) = \int_{\Ri^d} \Gamma_\mu (\cdot,y) \, f(y) \, dw(y)
		\]
		is the unique weak solution of $Lu = f$ in $\Ri^d$.
	\end{tabeleq}
\end{prop}

\begin{proof}
	By \cite[Theorem 1.3(v)]{CW} we know that $\Gamma_0(x,\cdot) \in L^p_\loc(\Ri^d)$ for all $p \in (0, \frac{d}{d-2})$ and $x \in \Ri^d$.
	The two statements then follow immediately from Propositions \ref{unique} and \ref{Gam poitwise bound}. 
\end{proof}

\subsection{Upper bounds for solutions}

We first prove a Caccioppoli's inequality for solutions to the equation $(L+i\tau)u=0$.

\begin{lemm} \label{caccioppoli}
Let $\tau \in \Ri$, $x_0 \in \Ri^d$, $R > 0$ and $B = B(x_0,R)$.
Let $u$ be a solution of $Lu + i\tau \, u = 0$ in $B$.
Then for every $\sigma \in (0,1)$ there exists a $C > 0$ such that 
\[
\int_{\sigma B} |\nabla u|^2 \, dw 
+ \int_{\sigma B} |u|^2 \, d\pi
+ \int_{\sigma B} |\tau| \, |u|^2 \, dw
\le \frac{C}{R^2} \, \int_B |u|^2 \, dw.
\]
\end{lemm}

\begin{proof}
Let $\eta \in C_c^\infty(B)$ be such that 
\[
\eta \ge 0,
\quad
\eta|_{\sigma B} = 1
\ca
|\nabla \eta| \le \frac{1}{(1-\sigma) \, R}.
\]

Using $\eta^2 \overline{u}$ as a test function, we have
\[
\int_B \eta^2 \, A \, \nabla u \cdot \nabla \overline{u} \, dw 
+ \int_B \eta^2 \, |u|^2 \, d\pi
+ i\tau \, \int_{B} \eta^2 \, |u|^2 \, dw
= - 2 \int_B \eta \, \overline{u} \, A \, \nabla u \cdot \nabla\eta \, dw.
\]
Consequently,
\begin{eqnarray*}
\Lambda^{-1} \int_{\sigma B} \eta^2 \, |\nabla u|^2 \, dw 
&+& \int_{\sigma B} \eta^2 \, |u|^2 \, d\pi
+ |\tau| \, \int_{\sigma B} |u|^2 \, dw\\
&\ls& \left| \int_B \eta^2 \, A \, \nabla u \cdot \nabla \overline{u} \, dw 
+ \int_B \eta^2 \, |u|^2 \, d\pi
+ i\tau \,  \int_{B} \eta^2 \, |u|^2 \, dw \right|
\\
&=& 2 \left| \int_B \eta \, \overline{u} \, A \, \nabla u \cdot \nabla\eta \, dw \right|
\\
&\le& 2 \Lambda \int_B |\eta \, \nabla u| \, |\overline{u} \, \nabla \eta| \, dw
\\
&\le& \epsilon \int_B \eta^2 \, |\nabla u|^2 \, dw  + \frac{\Lambda^2}{\epsilon} \int_B |u|^2 \, |\nabla \eta|^2 \, dw
\\
&\le& \epsilon \int_B \eta^2 \, |\nabla u|^2 \, dw + \frac{\Lambda^2}{\epsilon \, (1-\sigma)^2 \, R^2} \int_B u^2 \, dw
\end{eqnarray*}
for all $\epsilon > 0$.

Choosing a sufficiently small $\epsilon$ in the above inequality, our claim is justified.
\end{proof}

\begin{lemm} \label{next to main}
Let $x_0 \in \Ri^d$, $R > 0$ and $B = B(x_0,R)$.
Let $u$ be a solution of $Lu + i\tau \, u=0$ in $4B$.
Then for all $k \in \Ni$ there exists a $C > 0$ such that 
\[
\sup_{B} |u|
\le \frac{C}{\big( 1 + R \sqrt{\tau} \big)^k \, \big( 1 + R \, m_w(x_0,\mu) \big)^k} \, 
\left( \frac{1}{w(2B)} \, \int_{2B} |u|^2 \, dw \right)^{1/2}.
\]
\end{lemm}

\begin{proof}
Let $k \in \Ni$, $B = B(x_0,R)$ and $B_k = (1 + 2^{-k \, \lceil (k_0+1)/2 \rceil})B$, where $\lceil a \rceil$ denotes the smallest integer greater than $a$ for every $a\in \mathbb{R}$.

To begin with, we will show that $|u|^2$ is a sub-solution of $L_0$ in $4B$ (in the sense of Definition \ref{sub-sol L0}).
Let $\psi \in C_c^1(4B)$ be non-negative.
Direct calculations give
\begin{eqnarray*}
	\int_{4B} A \, \nabla(|u|^2) \cdot \nabla \psi \, dx
	&=& \int_{4B} A \, \nabla(u \, \overline{u}) \cdot \nabla \psi \, dx
	= \int_{4B} A \, \big( \overline{u} \, \nabla u + u \, \nabla \overline{u} \big) \cdot \nabla \psi \, dx
	\\
	&=& 2 \, \Re \int_{4B} \overline{u} \, A \, \nabla u \cdot \nabla \psi \, dx,
\end{eqnarray*}
where $\Re z$ denotes the real part of a complex number $z$.

On the other hand, by using $\overline{u} \, \psi$ as a test function, the hypothesis gives
\begin{eqnarray*}
	\int_{4B} \psi \, A \, \nabla u \cdot \nabla \overline{u} \, dx
	+ \int_{4B} \overline{u} \, A \, \nabla u \cdot \nabla \psi \, dx
	+ \int_{4B} |u|^2 \, \psi \, d\pi 
	+ i \tau \,  \int_{4B} |u|^2 \, dw
	= 0.
\end{eqnarray*}
Now we take the real parts of both sides to derive
\[
\int_{4B} \psi \, A \, \nabla u \cdot \nabla \overline{u} \, dx
+ \Re \int_{4B} \overline{u} \, A \, \nabla u \cdot \nabla \psi \, dx
+ \int_{4B} |u|^2 \, \psi \, d\pi 
= 0,
\]
or equivalently,
\begin{eqnarray*}
	\Re \int_{4B} \overline{u} \, A \, \nabla u \cdot \nabla \psi \, dx
	&=& - \int_{4B} \psi \, A \, \nabla u \cdot \nabla \overline{u} \, dx
	- \int_{4B} |u|^2 \, \psi \, d\pi 
	\\
	&\le& -\Lambda^{-1} \int_{4B} \psi \, |\nabla u|^2 \, dw
	- \int_{4B} |u|^2 \, \psi \, d\pi 
	\le 0,
\end{eqnarray*}
where we used \eqref{degenerate} in the second step.

Hence 
\[
\int_{4B} A \, \nabla(|u|^2) \cdot \nabla \psi \, dx \le 0
\]
and so $|u|^2$ is a sub-solution of $L_0$ in $4B$.

With this in mind we next apply Lemma \ref{u sub L0} to obtain
\[
\sup_B |u| 
\ls \left( \frac{1}{w(B_k)} \, \int_{B_k} |u|^2 \, dw \right)^{1/2}.
\]
Hence the claim is clear if $k=0$.

Next suppose $k \ge 1$.
By Lemma \ref{caccioppoli} one has
\[
\int_{B_k} |\nabla u|^2 \, dw 
+ \int_{B_k} |u|^2 \, d\pi 
\ls \frac{1}{R^2} \, \int_{B_{k-1}} |u|^2 \, dw.
\]

Next let $\eta \in C_c^\infty(B_{k-1})$ be such that 
\[
\eta|_{B_k} = 1,
\quad
\eta \le 1
\ca 
|\nabla \eta| \ls \frac{1}{R}.
\]
Applying Proposition \ref{norm equiv} to $u \, \eta$ yields
\[
\int_{B_k} m_w(\cdot,\mu)^2 \, |u|^2 \, dw
\ls \int_{B_{k-1}} |\nabla u|^2 \, dw 
+ \int_{B_{k-1}} |u|^2 \, d\pi + \frac{1}{R^2} \, \int_{B_{k-1}} |u|^2 \, dw,
\]
which in turn implies
\[
\int_{B_k} m_w(\cdot,\mu)^2 \, |u|^2 \, dw
\ls \frac{1}{R^2} \, \int_{B_{k-1}} |u|^2 \, dw.
\]
Combining this with Proposition \ref{crit}(iv) we yield
\[
\int_{B_k} |u|^2 \, dw 
\ls \frac{1}{\big( 1 + R \, m_w(x_0,\mu) \big)^{2/(k_0+1)}} \, \int_{B_{k-1}} |u|^2 \, dw.
\]
Iterating the above estimate $k \, \lceil (k_0+1)/2 \rceil$ times and using Lemma \ref{u sub L0} we arrive at 
\[
\int_{B} |u|^2 \, dw
\ls
\frac{C}{\big( 1 + R \, m_w(x_0,\mu) \big)^k} \, 
\left( \frac{1}{w(2B)} \, \int_{2B} |u|^2 \, dw \right)^{1/2}.
\]
Similar arguments together with Lemma \ref{caccioppoli} also gives
\[
\int_{B} |u|^2 \, dw
\ls
\frac{C}{\big( 1 + R \sqrt{\tau} \big)^k} \, 
\left( \frac{1}{w(2B)} \, \int_{2B} |u|^2 \, dw \right)^{1/2}.
\]
To conclude we combine these two estimates and yield the claim.
\end{proof}

\begin{prop} \label{sub main}
Let $\tau \in \Ri$.
Let $\Gamma_\mu(x,y,\tau)$ be the fundamental solution of $
L + i\tau$ in $\Ri^d$. Then for every $k \in \Ni$ there exists a $C=C(k) > 0$ such that 
\begin{equation}\label{upper bound for gammamutau}
|\Gamma_\mu(x,y,\tau)|
\le \frac{C}{\big( 1 + |x-y| \sqrt{\tau} \big)^k \, \big( 1 + |x-y| \, m_w(x,\mu) \big)^k} \frac{|x-y|^2}{w(B(x,|x-y|))}
\end{equation}
for all $x, y \in \Ri^d$ such that $x \ne y$.
\end{prop}

\begin{proof}
	Let $x, y \in \Ri^d$ be such that $x \ne y$. 
Set $R = |x-y|$ and $B = B(x,R/8)$.
By Proposition \ref{funda0},
\begin{equation}
\Gamma_0(z,y) 
\ls \frac{|z-y|^2}{w(B(z,|z-y|))}
\end{equation}
for all $z \in 2B$. 

Applying \eqref{doubling from w},
\begin{equation} 
\begin{aligned} \label{G0 est}
\Gamma_0(z,y) 
&\ls \frac{|z-y|^2}{R^2} \f{w(B(z, 8R))}{w(B(z,|z-y|))}\f{R^2}{w(B(z,8R))}\\
&\ls \Big(\frac{|z-y|}{R}\Big)^{2d-2}\f{R^2}{w(8B)}\\
&\ls \frac{R^2}{w(B)} 
\end{aligned}
\end{equation}
for all $z \in 4B$. 

Next observe that  $u(\cdot) := \Gamma_\mu(\cdot,y,\tau)$ is a weak solution of $(L + i\tau)u=0$ in $4B$.
It follows that for all $k \in \Ni$ one has
\begin{eqnarray*}
	\sup_{B} |u| 
	&\ls& \frac{1}{\big( 1 + R \, m_w(x,\mu) \big)^k} \, \left( \frac{1}{w(2B)} \, \int_{2B} |u|^2 \, dw \right)^{1/2}
	\\
	&\ls& \frac{1}{\big( 1 + R \, m_w(x,\mu) \big)^k} \, \left( \frac{1}{w(2B)} \, \int_{2B} |\Gamma_0(\cdot,y)|^2 \, dw \right)^{1/2}
	\\
	&\ls& \frac{1}{\big( 1 + R \, m_w(x,\mu) \big)^k} \, \frac{R^2}{w(B)},
\end{eqnarray*}
where we used Lemma \ref{next to main} in the first step, Proposition \ref{Gam poitwise bound} in the second step as well as \eqref{G0 est} in the last step.
\end{proof}

\begin{proof}[{\bf Proof of Theorem \ref{main}}] 
	We note that $\Gamma_\mu(\cdot,\cdot) \ge 0$ follows from Proposition \ref{Gam poitwise bound}.
	The upper bound \eqref{upper bound for gammamu} for $\Gamma_\mu(\cdot,\cdot)$ is obtained by setting $\tau = 0$ in Proposition \ref{sub main}.
\end{proof}

\section{Estimates on heat kernel $k_t(x,y)$} \label{Sect kernels}

This section is dedicated to proving Theorem \ref{main 2}. 

We recall some estimates on the kernel  $h_t(\cdot,\cdot)$ of $e^{-tL_0}$. See for example \cite[Section 3]{Dzi}. 

\begin{prop}\label{ht prop}
	The following properties hold for the kernel $h_t(x,y)$ of $e^{-tL_0}$.
	\begin{tabeleq}
		\item There exist constants $C, c > 0$ such that  
		\[
		\frac{1}{w(B(x,\sqrt{t}))} \, \exp\Big(-\f{|x-y|^2}{ct}\Big) \ls h_t(x,y) 
		\ls \frac{1}{w(B(x,\sqrt{t}))} \, \exp\Big(-\f{|x-y|^2}{ct}\Big)
		\]
		for all $x,y \in \Ri^d$ and $t > 0$.
		\item There exist constants $C> 0$ and $\gamma\in (0,1]$ such that
		\[
		|h_t(x,y) - h_t(x,z)| 
		\ls \frac{1}{w(B(x,\sqrt{t}))} \, \left( \frac{|y-z|}{\sqrt{t}} \right)^\gamma \Big[\exp\Big(-\frac{|x-y|^2}{ct}\Big)+\exp\Big(-\frac{|x-z|^2}{ct}\Big)\Big]
		\]
		for all $x,y \in \Ri^d$ and $t > 0$.
		\item There exists a constant $C > 0$ such that for all $k \in \Ni$ there exists a $c > 0$ satisfying
		\[
		|\D_t^k h_t(x,y)| \le \frac{c}{t^k \, w(B(x,\sqrt{t}))} \, \exp\Big(-\f{|x-y|^2}{ct}\Big)
		\]
		for all $x,y \in \Ri^d$ and $t > 0$.
		\item For every $x\in \Ri^d$,
		\[
		\int_{\Ri^d} h_t(x,y)\,dw(y)=1.
		\]
	\end{tabeleq}
\end{prop}

Since $\mu$ is a non-negative Radon measure, a perturbation formula asserts that 
\begin{equation} \label{k<h}
\begin{aligned}
0 \le k_t(x,y) &\le h_t(x,y)&\lesi \frac{1}{w(B(x,\sqrt{t}))} \, \exp\Big(-\f{|x-y|^2}{ct}\Big)
\end{aligned}
\end{equation}
for all $x, y \in \Ri^d$ and all $t>0$.

We are now ready to give the proof of Theorem \ref{main 2} (i).

\begin{proof}[{\bf Proof of Theorem \ref{main 2}} {\rm(i)}]

In what follows, let $x, y \in \Ri^d$ and $t > 0$. We divide the proof into two cases.

{\bf Case I: $|x-y| \ge \sqrt{t}$.}

By functional calculus,
\[
k_t(x,y) = C \, \int_\Ri e^{it\tau} \, \Gamma_\mu(x,y,\tau) \, d\tau.
\]
Hence Theorem \ref{main} gives
\begin{equation} \label{kt est 1}
k_t(x,y) \ls \frac{1}{\big( 1 + \sqrt{t} \, m_w(x,\mu) \big)^k \, w(B(x,\sqrt{t}))}
\end{equation}
for all $k \in \Ni$.
Combining this with \eqref{k<h} and Proposition \ref{ht prop} (i), we have
\[
k_t(x,y) \ls \frac{1}{\big( 1 + \sqrt{t} \, m_w(x,\mu) \big)^k \, w(B(x,\sqrt{t}))} \, \exp\Big(-\f{|x-y|^2}{ct}\Big)
\]
for all $k \in \Ni$.


{\bf Case II: $|x-y| < \sqrt{t}$.}

The semigroup $\{e^{-tL}\}_{t>0}$ can be extended to a holomorphic contraction semigroup on $L^2_w(\Ri^d)$.
Therefore for all $k \in \Ni$ there exists a $C > 0$ such that
\[
\|\D_t^k e^{-tL}\|_{L^2_w(\Ri^d) \to L^2_w(\Ri^d)}
\le \frac{C}{t^k}.
\]

Observe that 
\[
(\D_t k_t)(x,y) = \big( \D_t e^{-tL} \big) \, \big(k_t(\cdot,y)\big)(x).
\]
Consequently, we obtain
\[
\|(\D_t k_t)(x,\cdot)\|_{L^2_w(\Ri^d)}
\ls \frac{1}{t} \, \|k_t(x,\cdot)\|_{L^2_w(\Ri^d)} 
\ls \frac{1}{t \, w(B(x,\sqrt{t}))^{1/2}},
\]
where we used \eqref{k<h} and Proposition \ref{ht prop}(i) in the last step.
Using Schwartz's inequality,
\begin{equation} \label{Dtkt est 1}
|(\D_t k_t)(x,y)|
\ls \frac{1}{t \, w(B(x,\sqrt{t}))^{1/2} \, w(B(y,\sqrt{t}))^{1/2}}.
\end{equation}

Next we estimate $k_t$ as follows
\begin{eqnarray*}
k_t(x,y) 
&=& \int_{\Ri^d} \Gamma_\mu(x,z) \, (\D_t k_t)(z,y) \, dw(z)
\\
&\ls& \int_{\Ri^d} \frac{1}{\big( 1 + |x-z| \, m_w(x,\mu) \big)^N} \frac{|x-z|^2}{w(B(x,|x-z|))}  \frac{1}{t \, w(B(y,\sqrt{t}))^{1/2} \, w(B(z,\sqrt{t}))^{1/2}} \, dw(z)
\\
&=& \sum_{j \in \Ni} \int_{2^jB \setminus 2^{j-1}B} \cdots
+ \sum_{j \in -\Ni^*} \int_{2^jB \setminus 2^{j-1}B} \cdots
\\
&=:& I + II
\end{eqnarray*}
for all $N > 0$, where $B := B(x, \rho_w(x,\mu))$.

Next we estimate each term separately.

{\bf Term $I$}: 
We have
\begin{eqnarray*}
I
&\ls& \sum_{j \in \Ni} \int_{2^jB \setminus 2^{j-1}B} \frac{1}{\big( 1 + 2^j \big)^N} \frac{2^{2j} \, \rho_w(x,\mu)^2}{w(B(x,2^j \, \rho_w(x,\mu)))} 
\\
&& \times \frac{1}{t \, w(B(x,\sqrt{t}))^{1/2} \, w(B(y,\sqrt{t}))^{1/2}} \, \left( 1 + \frac{2^j \, \rho_w(x,\mu)}{\sqrt{t}} \right)^{2d} \, dw(z)
\\
&\le& \left( \frac{\rho_w(x,\mu)}{\sqrt{t}} \right)^2 \, \frac{1}{w(B(x,\sqrt{t}))^{1/2} \, w(B(y,\sqrt{t}))^{1/2}}
\\
&& \times \sum_{j \in \Ni} \int_{2^jB \setminus 2^{j-1}B} \frac{2^{2j} \, (1 + 2^j)^{2d}}{\big( 1 + 2^j \big)^N} \, \frac{1}{w(B(x,2^j \, \rho_w(x,\mu)))} \, dw(z)
\\
&\ls& \left( \frac{\rho_w(x,\mu)}{\sqrt{t}} \right)^2 \, \frac{1}{w(B(x,\sqrt{t}))^{1/2} \, w(B(y,\sqrt{t}))^{1/2}},
\end{eqnarray*}
where $N$ is chosen large enough and we use the facts that $\sqrt{t} \, m_w(x,\mu) \ge 1$, $|x-z| \sim 2^j \, \rho_w(x,\mu)$ and
\[
w(B(x,\sqrt{t})) \ls w(B(z,\sqrt{t})) \, \left( 1 + \frac{|x-z|}{\sqrt{t}} \right)^{2d} 
\]
due to the doubling property \eqref{doubling from w} of $w$.

{\bf Term $II$}: 
Observe that in this case $|x-z| < 2\rho_w(x,\mu) < 2 \sqrt{t}$.
So $w(B(x,\sqrt{t})) \sim w(B(z,\sqrt{t}))$ and we have
\begin{eqnarray*}
II
&\ls& \sum_{j \in -\Ni} \int_{2^jB \setminus 2^{j-1}B} \frac{2^{2j} \, \rho_w(x,\mu)^2}{w(B(x,2^j \, \rho_w(x,\mu)))} 
\times \frac{1}{t \, w(B(x,\sqrt{t}))^{1/2} \, w(B(y,\sqrt{t}))^{1/2}} \, dw(z)
\\
&\ls& \left( \frac{\rho_w(x,\mu)}{\sqrt{t}} \right)^2 \, \frac{1}{w(B(x,\sqrt{t}))^{1/2} \, w(B(y,\sqrt{t}))^{1/2}}.
\end{eqnarray*}

In sum we have proved that 
\[
k_t(x,y) 
\ls \left( \frac{\rho_w(x,\mu)}{\sqrt{t}} \right)^2 \, \frac{1}{w(B(x,\sqrt{t}))^{1/2} \, w(B(y,\sqrt{t}))^{1/2}}.
\]
Also note that when $|x-y| < \sqrt{t}$ and $\sqrt{t} \, m_w(x,\mu) \ge 1$, Proposition \ref{crit}(iii) implies
\[
\big( \sqrt{t} \, m_w(x,\mu) \big)^{-1}
\ls \big( 1 + \sqrt{t} \, m_w(y,\mu) \big)^{-1/(k_0+1)}.
\]
Keeping in mind these two estimates, we now invoke the symmetry of $k_t$ and use Proposition \ref{ht prop} to obtain
\begin{eqnarray*}
k_t(x,y)
&\ls& \frac{1}{w(B(x,\sqrt{t}))^{1/2} \, w(B(y,\sqrt{t}))^{1/2}} \, \exp\Big(-\f{|x-y|^2}{ct}\Big)
\\
&& \times \big( 1 + \sqrt{t} \, m_w(x,\mu) \big)^{-1/(k_0+1)} \, \big( 1 + \sqrt{t} \, m_w(y,\mu) \big)^{-1/(k_0+1)}.
\end{eqnarray*}
In turn this better estimate of $k_t$ (compared to \eqref{kt est 1}) implies a better estimate of $\D_t k_t$ (compared to \eqref{Dtkt est 1}).
Particularly one has
\begin{eqnarray*}
|(\D_t k_t)(x,y)|
&=& \left| \int_{\Ri^d} (\D_t k_{t/2})(x,z) \, k_{t/2}(z,y) \, dz \right|
\\
&\ls& \frac{1}{t \, w(B(x,\sqrt{t}))^{1/2} \, w(B(y,\sqrt{t}))^{1/2}}
\\
&& \times \big( 1 + \sqrt{t} \, m_w(x,\mu) \big)^{-1/(k_0+1)} \, \big( 1 + \sqrt{t} \, m_w(y,\mu) \big)^{-1/(k_0+1)}.
\end{eqnarray*}

By iterating the above estimates $N(k_0+1)$ times, we arrive at 
\begin{eqnarray*}
k_t(x,y)
&\ls& \frac{1}{w(B(x,\sqrt{t}))^{1/2} \, w(B(y,\sqrt{t}))^{1/2}} \, \exp\Big(-\f{|x-y|^2}{ct}\Big)
\\
&& \times \big( 1 + \sqrt{t} \, m_w(x,\mu) \big)^{-N} \, \big( 1 + \sqrt{t} \, m_w(y,\mu) \big)^{-N}
\end{eqnarray*}
and
\begin{eqnarray*}
	|(\D_t k_t)(x,y)|
	&=& \left| \int_{\Ri^d} (\D_t k_{t/2})(x,z) \, k_{t/2}(z,y) \, dz \right|
	\\
	&\ls& \frac{1}{t \, w(B(x,\sqrt{t}))^{1/2} \, w(B(y,\sqrt{t}))^{1/2}}
	\\
	&& \times \big( 1 + \sqrt{t} \, m_w(x,\mu) \big)^{-N} \, \big( 1 + \sqrt{t} \, m_w(y,\mu) \big)^{-N}.
\end{eqnarray*}

Thus the claim follows after applying the estimate
\[
\frac{1}{w(B(y,\sqrt{t}))} \, \exp\Big(-\f{|x-y|^2}{ct}\Big)
\ls \frac{1}{w(B(x,\sqrt{t}))} \, \exp\Big(-\f{|x-y|^2}{c't}\Big).
\]
This finishes the proof of Theorem \ref{main 2} (i).
\end{proof}

\bigskip

The proof of the item (iii) in Theorem \ref{main 2} will be given below.
\begin{proof}[{\bf Proof of Theorem \ref{main 2} (iii)}:]
	Let $t > 0$ and $x, y \in \Ri^d$.
	If $t \ge \rho_w(x,\mu)^2$ the claim follows at once from Theorem \ref{main 2}.
	Hence we need only to prove the theorem assuming that  $t <\rho_w(x,\mu)^2$.
	
	By Duhamel's formula we have
	\[
	h_t(x,y) - k_t(x,y) = \int_0^t\int_{\Ri^d} h_s(x,u) \, k_{t-s}(u,y) \, d\pi(u) \, ds.
	\]

	It follows from Proposition \ref{ht prop} (i) that
	\begin{eqnarray*}
		q_t(x,y)
		&\ls& \int_0^t \int_{\Ri^d} \frac{1}{w(B(x,\sqrt{s}))} \, e^{-\frac{|x-u|^2}{cs}} \, k_{t-s}(u,y) \, d\pi(u) \, ds
		\\
		&=:& I + II,
	\end{eqnarray*}
	where 
	\begin{eqnarray*}
		I &:=& \int_0^{t/2} \int_{\Ri^d} \frac{1}{w(B(x,\sqrt{s}))} \, \, e^{-\frac{|x-u|^2}{cs}} \, k_{t-s}(u,y) \, d\pi(u) \, ds
		\quad
		\mbox{and}
		\\
		II &:=& \int_{t/2}^t \int_{\Ri^d} \frac{1}{w(B(x,\sqrt{s}))} \, \, e^{-\frac{|x-u|^2}{cs}} \, k_{t-s}(u,y) \, d\pi(u) \, ds.
	\end{eqnarray*}
	Next we estimate each term separately.
	
	\noindent{\bf Term $I$}: The Gaussian upper bound of the kernel $k_t(\cdot,\cdot)$ allows us to get the bound
	\begin{align*}
		I 
		&= \int_0^{t/2} \int_{\Ri^d} \frac{1}{w(B(x,\sqrt{s})) \, w(B(y,\sqrt{t-s}))} \, \, e^{-\frac{|x-u|^2}{cs}} \, e^{-\frac{|y-u|^2}{c(t-s)}} \, d\pi(u) \, ds
		\\
		&\ls \int_0^{t/2} \int_{\Ri^d} \frac{1}{w(B(x,\sqrt{s})) \, w(B(y,\sqrt{t}))} \, \, e^{-\frac{|x-u|^2}{cs}} \, e^{-\frac{|y-u|^2}{ct}} \, d\pi(u) \, ds.
	\end{align*}
	
	Note that 
	\[
	e^{-\, \frac{|x-u|^2}{cs}} \, e^{-\, \frac{|y-u|^2}{ct}}
	\le e^{-\frac{|x-u|^2}{ct}} \, e^{-\frac{|y-u|^2}{ct}}
	\ls e^{-\f{|x-y|^2}{ct}}.
	\]
	Consequently,
	\begin{align*}
		I
		&\ls \frac{1}{w(B(y,\sqrt{t}))} \, \exp\Big(-\f{|x-y|^2}{ct}\Big) \, \int_0^{t/2} \int_{\Ri^d} \frac{1}{w(B(x,\sqrt{s}))} \, e^{-\frac{|x-u|^2}{cs}} \, d\pi(u) \, ds
		\\
		&\ls \frac{1}{w(B(y,\sqrt{t}))} \, \exp\Big(-\f{|x-y|^2}{ct}\Big) \, \int_0^{t/2} \frac{1}{s} \left(\frac{\sqrt{s}}{\rho_w(x,\mu)}\right)^\delta \, ds
		\\
		&\sim \frac{1}{w(B(y,\sqrt{t}))} \, \exp\Big(-\f{|x-y|^2}{ct}\Big)  \left(\frac{\sqrt{t}}{\rho_w(x,\mu)}\right)^\delta
		\\
		&\ls \frac{1}{w(B(x,\sqrt{t}))} \, \exp\Big(-\f{|x-y|^2}{ct}\Big) \, \left(\frac{\sqrt{t}}{\rho_w(x,\mu)}\right)^\delta,
	\end{align*}
	where we used Lemma \ref{Schwartz} in the second step and the fact that 
	\[
	\frac{1}{w(B(y,\sqrt{t}))} \, \exp\Big(-\f{|x-y|^2}{ct}\Big)
	\ls \frac{1}{w(B(x,\sqrt{t}))} \, \exp\left(-\frac{|x-y|^2}{c't}\right)
	\]
	in the last step.

	{\bf Term $II$}: Using a change of variables we can rewrite $II$ as follows:
	\begin{equation*}
		\int_0^{t/2} \int_{\Ri^d} \frac{1}{w(B(x,\sqrt{t-s})) \, w(B(y,\sqrt{s}))} \, \, e^{-\frac{|x-u|^2}{c(t-s)}} \, e^{-\frac{|y-u|^2}{cs}} \, d\pi(u) \, ds.
	\end{equation*}
	Arguing similarly to $I$ we conclude that 
	\[
	II 
	\ls \frac{1}{w(B(x,\sqrt{t}))} \, \exp\Big(-\f{|x-y|^2}{ct}\Big) \, \left(\frac{\sqrt{t}}{\rho_w(y,\mu)}\right)^\delta.
	\]
	On the other hand, from Proposition \ref{crit} (iii) and the fact that $t <\rho_w(x,\mu)^2$,
	\[
	\f{1}{\rho_w(y,\mu)}\lesi \f{1}{\rho_w(x,\mu)}\Big(1+\f{|x-y|}{\rho_w(x,\mu)}\Big)^{k_0}\lesi \f{1}{\rho_w(x,\mu)}\Big(1+\f{|x-y|}{t}\Big)^{k_0},
	\]
	which implies that
	\begin{equation} \label{wierd sim}
	\left( \frac{\sqrt{t}}{\rho_w(x,\mu)} \right)^{\delta} \, \exp\Big(-\f{|x-y|^2}{ct}\Big)
	\ls \left( \frac{\sqrt{t}}{\rho_w(y,\mu)} \right)^{\delta} \, \exp\Big(-\frac{|x-y|^2}{c't}\Big).
	\end{equation}
	Consequently,
	\[
	II 
	\ls \frac{1}{w(B(x,\sqrt{t}))} \, \exp\Big(-\f{|x-y|^2}{ct}\Big) \, \left(\frac{\sqrt{t}}{\rho_w(x,\mu)}\right)^\delta.
	\]

	Combining the estimates for $I$ and $II$, we arrive at the claim.
	
	This completes our proof.
\end{proof}

\bigskip

In order to prove  Theorem \ref{main 2} (ii), we need some technical ingredients. Set
\[
q_t(x,y) = h_t(x,y) - k_t(x,y)
\]
for all $x,y \in \Ri^d$ and $t > 0$. We have the following estimates on $q_t(x,y)$.

\begin{prop}\label{Proposition: holder of qt}
	For any $0<\theta<\min\{\gamma, \delta\}$ there exist $C$ and $c>0$ so that
	\begin{equation}\label{holder of qt}
	|q_{t}(x,y)-q_{t}(\overline{x},y)|\leq C\min\left\{\Big(\f{|x-\overline{x}|}{\rho_w(y,\mu)}\Big)^{\theta}, \Big(\f{|x-\overline{x}|}{\sqrt{t}}\Big)^{\theta}\right\}\f{1}{w(B(x,\sqrt{t}))}\exp\Big(-\f{|x-y|^2}{ct}\Big)
	\end{equation}
	for all $t>0$,  $|x-\overline{x}|<|x-y|/4$ and $|x-\overline{x}|<\rho_w(x,\mu)$.
\end{prop}
\begin{proof}
	By Duhamel's formula, we have
	\[
	\begin{aligned}
	q_{t}(x,y)-q_{t}(\overline{x},y)&=\int_0^t\int_{\mathbb{R}^d} (h_s(x,z)-h_s(\overline{x},z))k_{t-s}(z,y)d\pi(z)ds\\
	&=\int_0^{t/2}\int_{\mathbb{R}^d}\ldots+\int_{t/2}^t\int_{\mathbb{R}^d}\ldots:= I_1+I_2.
	\end{aligned}
	\]
	We now take care of $I_1$ first. To do this we write
	$$
	I_1=\int_0^{t/2}\int_{B(x,|x-y|/2)}\ldots+\int_0^{t/2}\int_{ B(x,|x-y|/2)^c}\ldots := I_{11}+I_{12}.
	$$
	Note that for $z\in B(x,|x-y|/2)$, $|z-y|\sim |x-y|$. Applying Proposition \ref{ht prop} (ii), Theorem \ref{main 2} and using  the fact that $t-s\sim t$ for  $s\in (0,t/2)$, we can bound the term $I_{11}$ as follows:
	$$
	\begin{aligned}
	I_{11}&\lesi\int_0^{t/2}\int_{B(x,|x-y|/2)}\Big(\f{|x-\overline{x}|}{\sqrt{s}}\Big)^{\theta}\f{1}{w(B(z,\sqrt{s}))}\Big[\exp\Big(-\f{|x-z|^2}{cs}\Big)+\exp\Big(-\f{|\overline{x}-z|^2}{cs}\Big)\Big]\\
	& \ \ \ \ \ \times \f{1}{w(B(y,\sqrt{t}))}\exp\Big(-\f{|x-y|^2}{ct}\Big)\Big(1+\f{\sqrt{t}}{\rho_w(y,\mu)}\Big)^{-\theta k_0-\theta}d\pi(z)ds\\
	&\lesi\int_0^{\rho_w(x,\mu)^2}\int_{B(x,2|x-\overline{x}|)}\ldots+\int_{\rho_w(x,\mu)^2}^{t/2}\int_{B(x,2|x-\overline{x}|)}\ldots:= J_1+J_2.
	\end{aligned}
	$$
	
	Note that $\rho_w(x,\mu)\sim \rho_w(\overline{x},\mu)$ for $|x-\overline{x}|\leq \rho_w(x,\mu)$. This, together with Lemma \ref{Schwartz} and $\delta>\theta$, gives
	$$
	\begin{aligned}
	J_1&\lesi \f{1}{w(B(y,\sqrt{t}))}\exp\Big(-\f{|x-y|^2}{ct}\Big)\Big(1+\f{\sqrt{t}}{\rho_w(y,\mu)}\Big)^{-\theta k_0-\theta}\int_0^{\rho_w(x,\mu)^2}\Big(\f{|x-\overline{x}|}{\sqrt{s}}\Big)^\theta\Big(\f{\sqrt{s}}{\rho_w(x,\mu)}\Big)^{\delta}\f{ds}{s}\\
	&\lesi \Big(\f{|x-\overline{x}|}{\rho_w(x,\mu)}\Big)^\theta\f{1}{w(B(y,\sqrt{t}))}\exp\Big(-\f{|x-y|^2}{ct}\Big)\Big(1+\f{\sqrt{t}}{\rho_w(y,\mu)}\Big)^{-\theta k_0-\theta}.
	\end{aligned}
	$$
	Owing to Proposition \ref{crit},
	$$
	\begin{aligned}
	J_1&\lesi \Big(\f{|x-\overline{x}|}{\rho_w(y,\mu)}\Big)^\theta\f{1}{w(B(y,\sqrt{t}))}\exp\Big(-\f{|x-y|^2}{ct}\Big)\Big(1+\f{|x-y|}{\rho_w(y,\mu)}\Big)^{\theta k_0}\Big(1+\f{\sqrt{t}}{\rho_w(y,\mu)}\Big)^{-\theta k_0-\theta}.
	\end{aligned}
	$$
    Using the inequality 
	$$
	\Big(1+\f{|x-y|}{\rho_w(y,\mu)}\Big)^{\theta k_0}\Big(1+\f{\sqrt{t}}{\rho_w(y,\mu)}\Big)^{-\theta k_0}\lesi \Big(1+\f{|x-y|}{\sqrt{t}}\Big)^{\theta k_0},
	$$
	we obtain that
	$$
	\begin{aligned}
	J_1&\lesi \Big(\f{|x-\overline{x}|}{\rho_w(y,\mu)}\Big)^\theta\f{1}{w(B(y,\sqrt{t}))}\exp\Big(-\f{|x-y|^2}{ct}\Big)\Big(1+\f{|x-y|}{\sqrt{t}}\Big)^{\theta k_0}\Big(1+\f{\sqrt{t}}{\rho_w(y,\mu)}\Big)^{-\theta}\\
	&\lesi \Big(\f{|x-\overline{x}|}{\rho_w(y,\mu)}\Big)^\theta\f{1}{w(B(y,\sqrt{t}))}\exp\Big(-\f{|x-y|^2}{c't}\Big)\Big(1+\f{\sqrt{t}}{\rho_w(y,\mu)}\Big)^{-\theta}.
	\end{aligned}
	$$
	Similarly, by Lemma \ref{Schwartz} and $N>N_0>\delta>\theta$, we have
	$$
	\begin{aligned}
	J_2&\lesi \f{1}{w(B(y,\sqrt{t}))}\exp\Big(-\f{|x-y|^2}{ct}\Big)\Big(1+\f{\sqrt{t}}{\rho_w(y,\mu)}\Big)^{-N}\int_{\rho_w(x,\mu)^2}^{t/2}\Big(\f{|x-\overline{x}|}{\sqrt{s}}\Big)^\theta\Big(\f{\sqrt{s}}{\rho_w(x,\mu)}\Big)^{N_0}
	\f{ds}{s}\\
	&\lesi \f{1}{w(B(y,\sqrt{t}))}\exp\Big(-\f{|x-y|^2}{ct}\Big)\Big(\f{|x-\overline{x}|}{\rho_w(x,\mu)}\Big)^\theta\Big(\f{\sqrt{t}}{\rho_w(x,\mu)}\Big)^{N_0-\theta}
	\Big(1+\f{\sqrt{t}}{\rho_w(y,\mu)}\Big)^{-N}\\
	&\lesi \f{1}{w(B(y,\sqrt{t}))}\exp\Big(-\f{|x-y|^2}{ct}\Big)\Big(\f{|x-\overline{x}|}{\rho_w(x,\mu)}\Big)^\theta\Big(\f{\rho_w(y,\mu)}{\rho_w(x,\mu)}\Big)^{N_0-\theta}
	\Big(1+\f{\sqrt{t}}{\rho_w(y,\mu)}\Big)^{-N}.
	\end{aligned}
	$$
	Applying Proposition \ref{crit},
	$$
	\begin{aligned}
	J_2
	&\lesi \f{1}{w(B(y,\sqrt{t}))}\exp\Big(-\f{|x-y|^2}{ct}\Big)\Big(\f{|x-\overline{x}|}{\rho_w(y,\mu)}\Big)^\theta\Big(1+\f{|x-y|}{\rho_w(y,\mu)}\Big)^{(N_0-\theta)k_0}
	\Big(1+\f{\sqrt{t}}{\rho_w(y,\mu)}\Big)^{-N}.
	\end{aligned}
	$$
	Using the following inequality
	\[
	\Big(1+\f{|x-y|}{\rho_w(y,\mu)}\Big)^{(N_0-\theta)k_0}
	\Big(1+\f{\sqrt{t}}{\rho_w(y,\mu)}\Big)^{-(N_0-\theta)k_0}\le \Big(1+\f{|x-y|}{\sqrt t}\Big)^{(N_0-\theta)k_0},
	\]
	and taking $N=(N_0-\theta)k_0 +\theta$, we obtain
	$$
		\begin{aligned}
	J_2&\lesi \Big(\f{|x-\overline{x}|}{\rho_w(y,\mu)}\Big)^\theta\f{1}{w(B(y,\sqrt{t}))}\exp\Big(-\f{|x-y|^2}{ct}\Big)\Big(1+\f{|x-y|}{\sqrt t}\Big)^{(N_0-\theta)k_0}\Big(1+\f{\sqrt{t}}{\rho_w(y,\mu)}\Big)^{-\theta}\\
	&\lesi \Big(\f{|x-\overline{x}|}{\rho_w(y,\mu)}\Big)^\theta\f{1}{w(B(y,\sqrt{t}))}\exp\Big(-\f{|x-y|^2}{c't}\Big) \Big(1+\f{\sqrt{t}}{\rho_w(y,\mu)}\Big)^{-\theta}.
		\end{aligned}
	$$
	Consequently,
	\[
	I_{11}\lesi \Big(\f{|x-\overline{x}|}{\rho_w(y,\mu)}\Big)^\theta\f{1}{w(B(y,\sqrt{t}))}\exp\Big(-\f{|x-y|^2}{c't}\Big) \Big(1+\f{\sqrt{t}}{\rho_w(y,\mu)}\Big)^{-\theta}.
	\]
	Arguing similarly we obtain
	\[
	I_{12}\lesi \Big(\f{|x-\overline{x}|}{\rho_w(y,\mu)}\Big)^\theta\f{1}{w(B(y,\sqrt{t}))}\exp\Big(-\f{|x-y|^2}{ct}\Big)\Big(1+\f{\sqrt{t}}{\rho_w(y,\mu)}\Big)^{-\theta}.
	\]
	Taking estimates $I_{11}$ and $I_{12}$ into account we conclude that
	\[
	\begin{aligned}
	I_1&\lesi \Big(\f{|x-\overline{x}|}{\rho_w(y,\mu)}\Big)^\theta\f{1}{w(B(y,\sqrt{t}))}\exp\Big(-\f{|x-y|^2}{ct}\Big)\Big(1+\f{\sqrt{t}}{\rho_w(y,\mu)}\Big)^{-\theta}\\
	&\lesi \min\left\{\Big(\f{|x-\overline{x}|}{\sqrt{t}}\Big)^\theta,\Big(\f{|x-\overline{x}|}{\rho_w(y,\mu)}\Big)^\theta\right\}\f{1}{w(B(y,\sqrt{t}))}\exp\Big(-\f{|x-y|^2}{ct}\Big),
	\end{aligned}
	\]
	where in the last inequality we used
	\[
	\Big(\f{|x-\overline{x}|}{\rho_w(y,\mu)}\Big)^\theta\Big(1+\f{\sqrt{t}}{\rho_w(y,\mu)}\Big)^{-\theta}\le \min\left\{\Big(\f{|x-\overline{x}|}{\sqrt{t}}\Big)^\theta,\Big(\f{|x-\overline{x}|}{\rho_w(y,\mu)}\Big)^\theta\right\}.
	\]
	It remains to take care of the term $I_2$. By a change of variable we can rewrite
	\[
	I_2=\int_{0}^{t/2}\int_{\mathbb{R}^d} (h_{t-s}(x,z)-h_{t-s}(\overline{x},z))k_{s}(z,y)d\pi(z)ds.
	\]
	By Proposition \ref{ht prop} (ii), Theorem \ref{main 2} and the fact that $t-s\sim t$ for $s\in (0,t/2]$,
	\[
	\begin{aligned}
	I_{2}&\lesi\int_0^{t/2}\int_{\Ri^d}\Big(\f{|x-\overline{x}|}{\sqrt{t}}\Big)^{\theta}\f{1}{w(B(z,\sqrt{t}))}\exp\Big(-\f{|x-z|^2}{ct}\Big)\\
	& \ \ \ \ \ \ \ \ \ \ \ \times \f{1}{w(B(y,\sqrt{s}))}\exp\Big(-\f{|z-y|^2}{cs}\Big)\Big(1+\f{\sqrt{s}}{\rho_w(y,\mu)}\Big)^{-N}d\pi(z)ds\\
	&\ +\int_0^{t/2}\int_{\Ri^d}\Big(\f{|x-\overline{x}|}{\sqrt{t}}\Big)^\theta\f{1}{w(B(z,\sqrt{t}))}\exp\Big(-\f{|\overline{x}-z|^2}{ct}\Big)\\
	& \ \ \ \ \ \ \ \ \ \ \  \times \f{1}{w(B(y,\sqrt{s}))}\exp\Big(-\f{|z-y|^2}{cs}\Big)\Big(1+\f{\sqrt{s}}{\rho_w(y,\mu)}\Big)^{-N}d\pi(z)ds\\
	&=:I_{21}+I_{22},
	\end{aligned}
	\]
	where $N>0$ will be fixed later.
	
	Note that for $s\in (0,t/2]$ we have
	$$
	\exp\Big(-\f{|x-z|^2}{ct}\Big)\exp\Big(-\f{|z-y|^2}{cs}\Big)\lesi \exp\Big(-\f{|x-y|^2}{c't}\Big)\exp\Big(-\f{|z-y|^2}{c''s}\Big).
	$$
	Inserting this into the expression of $I_{21}$ we obtain
	\[
	\begin{aligned}
	I_{21}& \lesi\Big(\f{|x-\overline{x}|}{\sqrt{t}}\Big)^\theta\f{1}{w(B(x,\sqrt{t}))}\exp\Big(-\f{|x-y|^2}{c't}\Big)\\
	& \ \ \times \int_0^{t/2}\int_{\Ri^d} \f{1}{w(B(y,\sqrt{s}))}\exp\Big(-\f{|z-y|^2}{c''s}\Big)\Big(1+\f{\sqrt{s}}{\rho_w(y,\mu)}\Big)^{-N}d\pi(z)ds.
	\end{aligned}
	\]
	If $t/2>\rho_w(y,\mu)$, then by Lemma \ref{Schwartz} we have
	\[
	\begin{aligned}
	\int_0^{t/2}\int_{\Ri^d} &\f{1}{w(B(y,\sqrt{s}))}\exp\Big(-\f{|z-y|^2}{c''s}\Big)\Big(1+\f{\sqrt{s}}{\rho_w(y,\mu)}\Big)^{-N}d\pi(z)ds\\
	&\lesi \int_0^{\rho_w(y,\mu)^2}\Big(\f{\sqrt{s}}{\rho_w(y,\mu)}\Big)^{\delta}\f{ds}{s}+ \int_{\rho_w(y,\mu)^2}^\vc\Big(\f{\sqrt{s}}{\rho_w(y,\mu)}\Big)^{N_0}\Big(\f{\sqrt{s}}{\rho_w(y,\mu)}\Big)^{-N}\f{ds}{s}\\
	&\lesi 1.
	\end{aligned}
	\]
	Hence,
	\[
	\begin{aligned}
	I_{21}&
	\lesi\Big(\f{|x-\overline{x}|}{\sqrt{t}}\Big)^\theta\f{1}{w(B(x,\sqrt{t}))}\exp\Big(-\f{|x-y|^2}{c't}\Big)\\
	&
	\lesi\min\left\{\Big(\f{|x-\overline{x}|}{\sqrt{t}}\Big)^\theta,\Big(\f{|x-\overline{x}|}{\rho_w(y,\mu)}\Big)^\theta\right\}\f{1}{w(B(x,\sqrt{t}))}\exp\Big(-\f{|x-y|^2}{c't}\Big).
	\end{aligned}
	\]
	If $t/2<\rho_w(y,\mu)$, taking $N=\delta-\theta$ then by Lemma \ref{Schwartz} we obtain 	
	\[
	\begin{aligned}
	I_{21}&\lesi \Big(\f{|x-\overline{x}|}{\sqrt{t}}\Big)^\theta\f{1}{w(B(x,\sqrt{t}))}\exp\Big(-\f{|x-y|^2}{c't}\Big)\int_0^{t/2}\Big(\f{\sqrt{s}}{\rho_w(y,\mu)}\Big)^{\delta}\Big(\f{\sqrt{s}}{\rho_w(y,\mu)}\Big)^{-\delta+\theta}\f{ds}{s}\\
	&\lesi \Big(\f{|x-\overline{x}|}{\sqrt{t}}\Big)^\theta\f{1}{w(B(x,\sqrt{t}))}\exp\Big(-\f{|x-y|^2}{ct}\Big)\Big(\f{\sqrt{t}}{\rho_w(y,\mu)}\Big)^\theta\\
	&\lesi \f{1}{w(B(x,\sqrt{t}))}\exp\Big(-\f{|x-y|^2}{ct}\Big)\Big(\f{|x-\overline{x}|}{\rho_w(y,\mu)}\Big)^\theta\\
	&
	\lesi\min\left\{\Big(\f{|x-\overline{x}|}{\sqrt{t}}\Big)^\theta,\Big(\f{|x-\overline{x}|}{\rho_w(y,\mu)}\Big)^\theta\right\}\f{1}{w(B(x,\sqrt{t}))}\exp\Big(-\f{|x-y|^2}{c't}\Big).
	\end{aligned}
	\]	
	By a  similar argument, we also have
	\[
	\begin{aligned}
	I_{22}&\lesi \f{1}{w(B(\overline{x},\sqrt{t}))}\exp\Big(-\f{|\overline{x}-y|^2}{ct}\Big)\min\left\{\Big(\f{|x-\overline{x}|}{\sqrt{t}}\Big)^\theta,\Big(\f{|x-\overline{x}|}{\rho_w(y,\mu)}\Big)^\theta\right\}\\
	&\lesi \f{1}{w(B(y,\sqrt{t}))}\exp\Big(-\f{|\overline{x}-y|^2}{ct}\Big)\min\left\{\Big(\f{|x-\overline{x}|}{\sqrt{t}}\Big)^\theta,\Big(\f{|x-\overline{x}|}{\rho_w(y,\mu)}\Big)^\theta\right\}\\
	&\lesi \f{1}{w(B(y,\sqrt{t}))}\exp\Big(-\f{|x-y|^2}{ct}\Big)\min\left\{\Big(\f{|x-\overline{x}|}{\sqrt{t}}\Big)^\theta,\Big(\f{|x-\overline{x}|}{\rho_w(y,\mu)}\Big)^\theta\right\}\\
	&\lesi \f{1}{w(B(x,\sqrt{t}))}\exp\Big(-\f{|x-y|^2}{ct}\Big)\min\left\{\Big(\f{|x-\overline{x}|}{\sqrt{t}}\Big)^\theta,\Big(\f{|x-\overline{x}|}{\rho_w(y,\mu)}\Big)^\theta\right\},
	\end{aligned}
	\]
	where in the third inequality we used the fact that $|\overline x-y|\sim |x-y|$.

	This completes our proof.
\end{proof}

We are now ready to give the proof for Theorem \ref{main 2} (ii).

\begin{proof}[{\bf Proof of  Theorem \ref{main 2} (ii)}:]
	Due to the kernel bound in (i) of Theorem \ref{main 2} we may assume that $|y-\overline y|<\sqrt t/4$. We now consider 2 cases.
	
	\noindent \textbf{Case 1: $|y-\overline{y}|<|x-y|/4$.}
	
	If $|y-\overline{y}|<\rho_w(y,\mu)$, then using the estimates in Proposition \ref{Proposition: holder of qt} and Proposition \ref{ht prop} we obtain \eqref{Holder p_t(x,y)}.
	
	Otherwise, if $|y-\overline{y}|\ge \rho_w(y,\mu)$, then applying (i) of Theorem \ref{main 2},
	\begin{equation}\label{eq- h-k proof}
     |k_{t}(x,y)-k_{t}(x,\overline{y})|\leq   \f{C}{w(B(x,\sqrt{t}))}\exp\Big(-\f{|x-y|^2}{ct}\Big)\Big[\Big(\f{\rho_w(y,\mu)}{\sqrt t}\Big)^{\theta}+\Big(\f{\rho_w(\overline y,\mu)}{\sqrt t}\Big)^{\theta}\Big].
	\end{equation}
	On the other hand, by Proposition \ref{crit},
	\[
	\rho_w(\overline y,\mu)\le C\rho_w(y,\mu)\Big(1+\f{|\overline y-y|}{\rho_w(\overline y,\mu)}\Big)^{\f{k_0}{k_0+1}}.
	\]
	This, along with $|y-\overline{y}|\ge \rho_w(y,\mu)$, implies that $|y-\overline{y}|\ge \rho_w(\overline y,\mu)$. Therefore, it follows from \eqref{eq- h-k proof} that 
	\[
	 |k_{t}(x,y)-k_{t}(x,\overline{y})|\leq   \Big(\f{|y-\overline y|}{\sqrt t}\Big)^{\theta}\f{C}{w(B(x,\sqrt{t}))}\exp\Big(-\f{|x-y|^2}{ct}\Big),
	\]
	which proves \eqref{Holder p_t(x,y)}.
	
	\bigskip
	
	\noindent \textbf{Case 2: $|y-\overline{y}|\ge |x-y|/4$.}
	
	We borrow some ideas in \cite{DZ2} to write
	\[
	\begin{aligned}
	k_{t}(x,y)-k_{t}(x,\overline{y})&=\int_{\Ri^d}k_{t/2}(x,z)\Big[k_{t/2}(z,y)-k_{t/2}(z,\overline y)\Big]dw(z)\\
	&=\int_{|y-z|\ge 4|\overline y-y|}\ldots +\int_{|y-z|< 4|\overline y-y|}\ldots\\
	&=E_1 + E_2.
	\end{aligned}
	\]
	We can apply the estimate in Case 1 to dominate the term $E_1$ by
	\[
	\Big(\f{|y-\overline y|}{\sqrt t}\Big)^{\theta}\int_{\Ri^d}|k_{t/2}(x,z)|\f{C}{w(B(y,\sqrt{t}))}\exp\Big(-\f{|z-y|^2}{ct}\Big)dw(z).
	\]
	Owing the Gaussian upper bound of $k_t$ we further obtain
	\[
	\begin{aligned}
	|E_1|&\lesi \Big(\f{|y-\overline y|}{\sqrt t}\Big)^{\theta}\int_{\Ri^d}\f{1}{w(B(x,\sqrt{t}))}\exp\Big(-\f{|x-z|^2}{ct}\Big)\f{1}{w(B(y,\sqrt{t}))}\exp\Big(-\f{|z-y|^2}{ct}\Big)dw(z).
	\end{aligned}
	\]
	By using the following inequality
	\[
	\exp\Big(-\f{|x-z|^2}{ct}\Big)\exp\Big(-\f{|z-y|^2}{ct}\Big)\lesi \exp\Big(-\f{|x-y|^2}{c't}\Big)\exp\Big(-\f{|z-y|^2}{c''t}\Big),
	\]
	we arrive at
	\[
	\begin{aligned}
	|E_1|&\lesi \Big(\f{|y-\overline y|}{\sqrt t}\Big)^{\theta}\f{1}{w(B(x,\sqrt{t}))}\exp\Big(-\f{|x-y|^2}{c't}\Big)\int_{\Ri^d} \f{1}{w(B(y,\sqrt{t}))}\exp\Big(-\f{|z-y|^2}{c''t}\Big)dw(z)\\
	&\lesi \Big(\f{|y-\overline y|}{\sqrt t}\Big)^{\theta}\f{1}{w(B(x,\sqrt{t}))}\exp\Big(-\f{|x-y|^2}{c't}\Big).
	\end{aligned}
	\]
	It remains to evaluate the term $E_2$. By invoking the Gaussian upper bound of $k_t$, we have
	\[
	\begin{aligned}
	|E_2|&\lesi \int_{B(y,4|\overline y -y|)}\f{1}{w(B(x,\sqrt{t}))}\exp\Big(-\f{|x-z|^2}{ct}\Big)\f{1}{w(B(y,\sqrt{t}))}\exp\Big(-\f{|z-y|^2}{ct}\Big)dw(z)\\
	&  \ \ \ \ +\int_{B(y,4|\overline y -y|)}\f{1}{w(B(x,\sqrt{t}))}\exp\Big(-\f{|x-z|^2}{ct}\Big)\f{1}{w(B(\overline y,\sqrt{t}))}\exp\Big(-\f{|z-\overline y|^2}{ct}\Big)dw(z)\\
	&\lesi \int_{B(y,4|\overline y -y|)}\f{1}{w(B(x,\sqrt{t}))}\exp\Big(-\f{|x-y|^2}{c't}\Big)\f{1}{w(B(y,\sqrt{t}))} dw(z)\\
	&  \ \ \ \ +\int_{B(y,4|\overline y -y|)}\f{1}{w(B(x,\sqrt{t}))}\exp\Big(-\f{|x-\overline y|^2}{ct}\Big)\f{1}{w(B(\overline y,\sqrt{t}))}dw(z),
	\end{aligned}
	\]
	where in the last inequality we used
	\[
	\exp\Big(-\f{|x-z|^2}{ct}\Big)\exp\Big(-\f{|z-y|^2}{ct}\Big)\lesi \exp\Big(-\f{|x-y|^2}{c't}\Big) 
	\] 
	and
	\[
	\exp\Big(-\f{|x-z|^2}{ct}\Big)\exp\Big(-\f{|z-\overline y|^2}{ct}\Big)\lesi \exp\Big(-\f{|x-\overline y|^2}{c't}\Big). 
	\]
	Since $|\overline y-y|<\sqrt t$,  $w(B(y,\sqrt{t}))\sim w(B(y,\sqrt{t}))$  and 
	\[
	\exp\Big(-\f{|x-\overline y|^2}{ct}\Big)\sim \exp\Big(-\f{|x-y|^2}{ct}\Big).
	\] 
	We thus obtain
	\[
	\begin{aligned}
	|E_2|
	&\lesi \int_{B(y,4|\overline y -y|)}\f{1}{w(B(x,\sqrt{t}))}\exp\Big(-\f{|x-y|^2}{c't}\Big)\f{1}{w(B(y,\sqrt{t}))} dw(z)\\
	&= \f{1}{w(B(x,\sqrt{t}))}\exp\Big(-\f{|x-y|^2}{c't}\Big)\f{w(B(y,4|\overline y -y|))}{w(B(y,\sqrt{t}))}\\
	&\lesi \Big(\f{|\overline y -y|}{\sqrt t}\Big)^\beta\f{1}{w(B(x,\sqrt{t}))}\exp\Big(-\f{|x-y|^2}{c't}\Big)\\
	&\lesi \Big(\f{|\overline y -y|}{\sqrt t}\Big)^\theta\f{1}{w(B(x,\sqrt{t}))}\exp\Big(-\f{|x-y|^2}{c't}\Big),
	\end{aligned}
	\]
	where in the third inequality we used \eqref{(RD)} and in the last inequality we used $\beta\ge 2>\theta$.

	This completes our proof.
\end{proof}

\section{Maximal function characterization for Hardy spaces $h^{p,q}_{at,\rho_w}(\Ri^d,w)$}

This section is dedicated to proving Theorem \ref{mainthm 3}.
\subsection{Local Hardy spaces}
We recall the notion of atomic Hardy spaces in \cite{YZ2}.
\begin{ddefi}
	Let $p\in (\f{n}{n+1},1], q\in [1,\vc]\cap (p,\vc]$ and $\ell>0$. A function $a$ is called a local $(p,q)_\ell$-atom associated to the ball $B(x_0,r)$ if
		\begin{enumerate}[{\rm (i)}]
			\item ${\rm supp}\, a\subset B(x_0,r)$;
			\item $\|a\|_{L^q(\Ri^d)_w}\leq w(B(x_0,r))^{1/q-1/p}$;
			\item $\int a\, dw =0$ if $r<\ell$.
		\end{enumerate}
\end{ddefi}

The local Hardy spaces are defined as follows.
\begin{ddefi}\label{defnHardyspaces}
		Let $p\in (\f{n}{n+1},1], q\in [1,\vc]\cap (p,\vc]$ and $\ell>0$. 
		The local Hardy space $h^{p,q}_{\ell,at}(\Ri^d,w)$ is defined to be the completion of the set of all $f=\sum_j \lambda_ja_j$ in $L^2_w(\Ri^d)$ under the norm
		$$\|f\|_{h^{p,q}_{\ell,at}(\Ri^d,w)}=\inf\left\{\sum_{j}|\lambda_j|^p: \, f=\sum_j \lambda_ja_j\right\},
		$$
		where $\{a_j\}_{j\in \mathbb{N}}$ are local $(p,q)_\ell$-atoms and $\{\lambda_j\}_{j\in \mathbb{N}}\subset \mathbb{C}$ such that $\sum_{j}|\lambda_j|^p<\vc$.
\end{ddefi}
It was proved in \cite{YZ} that $h^{p,q}_{\ell,at}(\Ri^d,w)=h^{p,r}_{\ell,at}(\Ri^d,w)$ for all $\f{n}{n+1}<p\leq 1$ , $q,r\in [1,\vc]\cap (p,\vc]$ and $\ell>0$. For this reason, we define the local Hardy spaces $h^p_\ell(\Ri^d,w)$ with $\f{n}{n+1}<p\leq 1$ and  $\ell>0$ to be any space $h^{p,r}_{\ell,at}(\Ri^d,w)$ with $q\in [1,\vc]\cap (p,\vc]$.

We recall the following result in \cite[Theorem 2.10]{BDK}.
\begin{thrm}\label{thm1-local hardy spaces basic property}
		Let $\f{n}{n+1}<p\leq 1$ , $q\in [1,\vc)\cap (p,\vc)$ and $\ell\in \mathbb{R}$.  If $f \in h^p_\ell(\Ri^d,w)$ is supported in a ball $B$ with radius of $r_B\ge \ell$, then there exist a number $c_0$, a sequence of numbers $\{\lambda_j\}_{j\in \mathbb{N}}$, and a sequence $\{a_j\}_{j\in \mathbb{N}}$ of local  $(p,q)_\ell$ atoms  such that for each $j$, $a_j$ is supported in $c_0B$ such that $f=\sum_{j=1}^\vc \lambda_j a_j$ and
			\begin{equation}
			\label{eq-compact hardy functions}
			\|f\|^p_{h^p_\ell(\Ri^d,w)}\sim \sum_{j=1}^\vc|\lambda_j|^p.
			\end{equation}
			If $f \in h^p_\ell(\Ri^d,w)\cap C(\Ri^d)$, then the statement is also  true with $q=\vc$.
\end{thrm}
\subsection{Some estimates on Hardy spaces $h^{p,q}_{at, \rho_w}$}
Let $\{B_j\}_{j\in \mathbb{N}}$ and $\{\psi_j\}_{j\in \mathbb{N}}$ be families of balls and functions in Lemma \ref{xalpha psialpha}.  From Proposition \ref{crit} there exists a $C_{p_w}$ such that 
\[
C_{p_w}^{-1}\rho_w(x,\mu)\le \rho_w(y,\mu)\le C_{p_w}\rho_w(x,\mu) \ \ \ \text{whenever $|x-y|<\rho_w(x,\mu)$}.
\]
We define $B^* = 4c_0 B$, where $c_0$ is the constant in Theorem \ref{thm1-local hardy spaces basic property}.

We first prove the following result which gives a localized maximal function estimate.

\begin{lemm}\label{Cor1: Hardy spaces}
	Let $\f{n}{n+\theta}<p\leq 1$ with $0<\theta<\min\{\delta,\gamma\}$ and $q\in (p,\vc]\cap [1,\vc]$, where $\delta$ is the constant in \eqref{(M1)} and $\gamma$ is the constant in Proposition \ref{ht prop}(ii). Then there exists a $C>0$ such that for any $0<\epsilon\leq 1$, we have
	$$
	\Big\|\sup_{0< t\leq \epsilon \rho_w(x_j,\mu)}|e^{-t\mathscr L_0}(f\psi_j )(x)|\,\Big\|^p_{L^p_w(X\backslash B^*_j)}\le C \gamma^{\theta p}\sum_{j\in \mathcal{I}_j}\|f\psi_j\|^p_{h^{p,q}_{at,\rho_w,\epsilon}(\Ri^d)}
	$$
	for all $f\in h^{p,q}_{at,\rho_w,\epsilon}(\Ri^d,w)$.
\end{lemm}
\begin{proof}
	From Corollary 3.4 in \cite{BDK}, we have
	\[
	\Big\|\sup_{0< t\leq \epsilon \rho_w(x_j,\mu)}|e^{-t\mathscr L_0}(f\psi_j )(x)|\,\Big\|^p_{L^p_w(X\backslash B^*_j)}\le C \gamma^{\theta p}\sum_{i\in \mathcal{I}_j}\|f\psi_i\|^p_{h^{p,q}_{at,\rho_w,\epsilon}(\Ri^d)}
	\]
	for each $j\in \mathbb{N}$, where $\mathcal{I}_j=\{i\in \mathbb{N}: B_i\cap B_j\ne \emptyset\}$. From Proposition \ref{xalpha psialpha}, it is easy to see that the cardinality of $\mathcal{I}_j$ is uniformly bounded by a constant for every $j\in \mathbb{N}$. Therefore, summing the above inequality for all $j\in \mathbb{N}$ we obtain the desired estimate.
	
	This completes our proof.
\end{proof}

The following two results are just direct consequences of Lemma 3.5 and Theorem 3.1 in \cite{BDK}.
\begin{lemm}\label{Lem2: Hardy spaces}
	Let $\f{n}{n+\theta}<p\leq 1$ with $0<\theta<\min\{\delta,\gamma\}$ and $q\in (p,\vc]\cap [1,\vc]$. Then,  for any $0<\epsilon\leq 1$, we have
	\begin{equation}\label{eq1-Lem2Hardyspaces}
	\Big\|\sum_{j\in \mathbb N}\sup_{0< t\leq [\epsilon \rho_w(x,\mu)]^2}|\psi_j(x)e^{-t\mathscr L_0}f(x) -e^{-t\mathscr L_0}(f\psi_j )(x)|\,\Big\|^p_{L^p_w(\Ri^d)}\lesi \epsilon^{\theta p}\|f\|^p_{h^{p,q}_{at,\rho_w,\epsilon}(\Ri^d,w)}
	\end{equation}
	for all $f\in h^{p,q}_{at,\rho_w,\epsilon}(\Ri^d,w)$.
\end{lemm}

\begin{thrm}\label{thm- H rho and H rad}
	Let $\f{n}{n+\theta}<p\leq 1$ with $0<\theta<\min\{\delta,\gamma\}$  and $q\in [1,\vc]\cap (p,\vc]$. Then, for any $\displaystyle f\in L^2_w(\Ri^d)$,  we have
	\[
	\Big\|\sup_{0<t<\rho_w(x,\mu)^2}|e^{-t\mathscr L_0}f(x)| \Big\|_{L^p_w(\Ri^d)}\sim \|f\|_{h^{p,q}_{at, \rho_w}(\Ri^d,w)}.
	\]
	
\end{thrm}
\subsection{Proof of Theorem \ref{mainthm 3}}
We now give the proof for Theorem \ref{mainthm 3}. In order to do this, we split the proof into 2 steps.

\noindent\textbf{Step 1: $h^{p,q}_{at,\rho_w}(\Ri^d,w)\hookrightarrow H^p_{L}(\Ri^d,w)$.}

\noindent\textbf{Step 2: $ H^p_{L}(\Ri^d,w)\hookrightarrow h^{p,q}_{at,\rho_w}(\Ri^d,w)$.}

\medskip	
\begin{proof}[{\bf Proof of Step 1}.]	
	We first prove the continuous embedding $h^{p,q}_{at,\rho}(\Ri^d,w)\hookrightarrow H^p_{L}(\Ri^d,w)$. Since the space $L^2(\Ri^d)$ is dense in both $h^{p,q}_{at,\rho_w}(\Ri^d,w)$ and $H^p_{L}(\Ri^d,w)$,  it suffices to show that  $h^{p,q}_{at,\rho_w}(\Ri^d)\cap  H^p_{L}(\Ri^d,w)$. Since $\mathcal M_L$ is dominated by the Hardy--Littlewood maximal function $\mathcal{M}$ (see for example \cite{DM}), $\mathcal M_L$ is bounded on $L^2_w(\Ri^d)$. Therefore, it suffices to prove that 
	\begin{align}\label{maximal atoms}
	\|\mathcal M_L a\|^p_{L^p}\leq C
	\end{align}
	for all $(p,q,\rho)$-atoms associated to balls $B=B(x_0,r)$.
	
	To prove \eqref{maximal atoms}, we first write
	$$
	\|\mathcal M_L a\|^p_{L^p_w}\leq \|\mathcal M_L a\|^p_{L^p_w(4B)}+\|\mathcal M_L a\|^p_{L^p_w(\Ri^d\backslash 4B)}: =I_1+I_2.
	$$
	Using  H\"older's inequality and the $L^q_w$-boundedness of $\mathcal{M}_L$ we can dominate $I_1$ by a constant. So, it remains to consider the contribution of $I_2$. To do this, we consider two cases.
	
	\textbf{Case 1: $\rho_w(x_0,\mu)/4 \leq r\leq \rho_w(x_0,\mu)$}.
	
	Using Theorem \ref{main 2} with $N=1$, we have
	\[
	\begin{aligned}
	I_2&\lesi \int_{\Ri^d\backslash 4B}\sup_{t>0}\Big[\int_B \f{1}{w(B(y,\sqrt{t}))}\exp\Big(-\f{|x-y|^2}{ct}\Big)\Big(1+\f{\sqrt t}{\rho_w(y,\mu)}\Big)^{-1}|a(y)|dw(y)\Big]^pdw(x)\\
	&\lesi \int_{\Ri^d\backslash 4B}\sup_{t>0}\Big[\int_B \f{1}{w(B(y,|x-y|))}\exp\Big(-\f{|x-y|^2}{ct}\Big)\Big(1+\f{\sqrt t}{\rho_w(y,\mu)}\Big)^{-1}|a(y)|dw(y)\Big]^pdw(x).
	\end{aligned}
	\]
	
	Since $y\in B(x_0,r)$ with $r\sim \rho_w(x_0,\mu)$, by Proposition \ref{crit} $\rho_w(y,\mu)\sim r$. Hence,
	\[
	\exp\Big(-\f{|x-y|^2}{ct}\Big)\Big(1+\f{\sqrt t}{\rho_w(y,\mu)}\Big)^{-1}\lesi \Big(\f{r}{|x-y|}\Big).
	\]
	Consequently,
	\[
	\begin{aligned}
	I_2&\lesi \int_{\Ri^d\backslash 4B} \Big[\int_B \f{1}{w(B(y,|x-y|))}\f{r}{|x-y|} |a(y)|dw(y)\Big]^pdw(x)\\
	&\sim \int_{\Ri^d\backslash 4B} \Big[\int_B \f{1}{w(B(x_0,|x-y_0|))} \f{r}{|x-y_0|} |a(y)|dw(y)\Big]^pdw(x)\\
	&\lesi \|a\|_{L^1}^p\int_{\Ri^d\backslash 4B} \Big[ \f{1}{w(B(x_0,|x-y_0|))} \f{r}{|x-y_0|}  \Big]^pdw(x)\\
	&\lesi w(B)^{p-1}\int_{\Ri^d\backslash 4B} \Big[ \f{1}{w(B(x_0,|x-y_0|))} \f{r}{|x-y_0|}  \Big]^pdw(x).
	\end{aligned}
	\]
	By a simple calculation, we come up with
	\[
	\int_{\Ri^d\backslash 4B} \Big[ \f{1}{w(B(x_0,|x-y_0|))} \f{r}{|x-y_0|}  \Big]^pdw(x)\lesi w(B)^{1-p}
	\]
	as long as $\f{n}{n+1}<p\le 1$.
	
	It follows that $I_2\lesi 1$. Hence, \eqref{maximal atoms} is proved.

	\textbf{Case 2: $r<\rho_w(x_0,\mu)/4$}.
	
	Observe that
	\[
	\begin{aligned}
	I_2&\lesi \int_{\Ri^d\backslash 4B}\sup_{0<t\leq 4r^2} \Big|\int_B k_t(x,y)a(y)dw(y)\Big|^pdw(x)+ \int_{\Ri^d\backslash 4B}\sup_{ t\ge 4r^2}\Big|\int_Bk_t(x,y)a(y)dw(y)\Big|^pdw(x)\\
	&=I_{21}+I_{22}.
	\end{aligned}
	\]
	By Theorem \ref{main 2},
	\[
	\begin{aligned}
	I_{21}&\lesi \int_{\Ri^d\backslash 4B}\sup_{0<t<4r^2}\Big[\int_B \f{1}{w(B(y,|x-y|))} \f{t}{|x-y|}|a(y)|dw(y)\Big]^pdw(x).
	\end{aligned}
	\]

	Arguing similarly to the estimate of $I_2$ in Case 1, we have 
	\[
	I_{21}\lesi 1.
	\] 
	
	To take care of $I_{22}$ we use the cancellation property of $a$ to arrive at
	\begin{equation*}\label{eqI22}
	I_{22}=\int_{\Ri^d\backslash 4B}\sup_{ t\ge 4r^2} \Big|\int_B [k_t(x,y)-{p}_t(x,x_0)]a(y)dw(y)\Big|^pdw(x).
	\end{equation*}
	
	Owing to Theorem \ref{main 2},
	\begin{equation}\label{eqI22}
	\begin{aligned}
	I_{22}&\lesi \int_{\Ri^d\backslash 4B}\sup_{ t\ge 4r^2} \Big|\int_B \Big(\f{|y-x_0|}{\sqrt t}\Big)^\theta \f{1}{w(B(y,\sqrt t))}\exp\Big(-\f{|x-y|^2}{ct}\Big)|a(y)|dw(y)\Big|^pdw(x)\\
	&\lesi \int_{\Ri^d\backslash 4B} \Big|\int_B \Big(\f{|y-x_0|}{|x-y|}\Big)^\theta \f{1}{w(B(y,|x-y|))} |a(y)|dw(y)\Big|^pdw(x)\\
	&\lesi \int_{\Ri^d\backslash 4B} \Big|\int_B \Big(\f{r}{|x-x_0|}\Big)^\theta \f{1}{w(B(x_0,|x-x_0|))} |a(y)|dw(y)\Big|^pdw(x).
	\end{aligned}
	\end{equation}
	At this stage, employing the argument used in the estimate of $I_2$ in Case 1, we also obtain 
	\[
	I_{22}\lesi 1,
	\]
	provided that $p>n/(n+\theta)$. 
	
	Therefore, this completes the proof of  \textbf{Step 1}.
\end{proof}

In order to prove \textbf{Step 2}, we need the following estimates.
\begin{lemm}\label{Lem3: Hardy spaces}
	Let $\f{n}{n+\theta}<p\leq 1$ and $q\in (p,\vc]\cap [1,\vc]$. Then there exists a $\kappa > 0$ such that for any $0<\epsilon\leq 1$, we have
	\begin{equation}\label{eq1-Lem3Hardy}
	\Big\|\sup_{0< t\leq [\epsilon \rho_w(x,\mu)]^2}|(e^{-tL}-e^{-tL_0}f(x)|\,\Big\|^p_{L^p_w(\Ri^d)}\lesi \epsilon^{\kappa}\|f\|^p_{h^{p,q}_{at,\rho_w,\epsilon}(\Ri^d,w)}
	\end{equation}
	for all $f\in h^{p,q}_{at,\rho_w,\epsilon}(\Ri^d,w)$.
\end{lemm}
\begin{proof}
	Observe that
	\[
	\sup_{0< t\leq [\epsilon \rho_w(x,\mu)]^2}|(e^{-tL}-e^{-tL_0}f(x)|\lesi \sup_{t>0}|e^{-tL}f(x)|\lesi \mathcal{M}f(x),
	\]
	where $\mathcal M$ is the Hardy--Littlewood maximal function.

	It follows that the operator
	\[
	f \mapsto \sup_{0< t\leq [\epsilon \rho_w(x,\mu)]^2}|(e^{-tL}-e^{-tL_0}f(x)|
	\]
	is bounded on $L^2_w(\Ri^d)$. Hence, it suffices to prove \eqref{eq1-Lem3Hardy} for all $(p,q,\rho_w,\epsilon)$ atoms. Let  $a$ be $(p,q,\rho_w,\epsilon)$ atom associated to a ball $B:=B(x_0,r)$. We write
	$$
	\begin{aligned}
	\Big\|\sup_{0< t\leq [\epsilon \rho_w(x,\mu)]^2}|(e^{-tL}-e^{-tL_0}f(x)|\,\Big\|^p_{L^p_w(\Ri^d)}&\leq \Big\|\sup_{0< t\leq [\epsilon \rho_w(x,\mu)]^2}|(e^{-tL}-e^{-tL_0}f(x)|\,\Big\|^p_{L^p_w(4B)}\\
	& \  \  \ +\Big\|\sup_{0< t\leq [\epsilon \rho_w(x,\mu)]^2}|(e^{-tL}-e^{-tL_0}f(x)|\,\Big\|^p_{L^p_w(\Ri^d\backslash 4B)}\\
	&=I_1+I_2.
	\end{aligned}
	$$
	Using \eqref{thm- difference 2 heat kernels}, H\"older's inequality and the $L^q_w$-boundedness of the Hardy-Littlewood maximal function $\mathcal{M}$, we get that
	$$
	\begin{aligned}
	I_1&\lesi \int_{4B}\Big[\sup_{0< t\leq [\epsilon \rho_w(x,\mu)]^2}\Big(\f{\sqrt{t}}{\rho_w(x,\mu)}\Big)^{\delta}\int_B\f{1}{w(B(x,\sqrt{t}))}\exp\Big(-\f{|x-y|^2}{ct}\Big)|a(y)|dw(y)\Big]^pdw(x)\\
	&\lesi \epsilon^{p\delta}w(B)^{1-p/q} \Big[\int_{4B}\Big[\sup_{0< t\leq [\epsilon \rho_w(x,\mu)]^2}\int_B\f{1}{w(B(x,\sqrt{t}))}\exp\Big(-\f{|x-y|^2}{ct}\Big)|a(y)|dw(y)\Big]^qdw(x)\Big]^{p/q}\\
	&\lesi \epsilon^{p\delta}w(B)^{1-p/q} \Big[\int_{4B}\Big[\mathcal{M}(|a|)(x)\Big]^qdw(x)\Big]^{p/q}\\
	&\lesi \epsilon^{p\delta}.
	\end{aligned}
	$$
	
	In order to take care of $I_2$, we consider the following two cases.
	
	\textbf{Case 1: $\epsilon\rho_w(x_0,\mu)/4\leq r\leq \epsilon\rho_w(x_0,\mu)$}.
	
	By \eqref{thm- difference 2 heat kernels} again, 
	$$
	\begin{aligned}
	I_2&=\int_{\Ri^d\backslash 4B}\Big[\sup_{0< t\leq [\epsilon \rho_w(x,\mu)]^2}\int_B \Big(\f{\sqrt{t}}{\rho_w(x,\mu)}\Big)^{\delta}\f{1}{w(B(x,\sqrt{t}))}\exp\Big(-\f{|x-y|^2}{ct}\Big)|a(y)|dw(y)\Big]^pdw(x)\\
	&\lesi \epsilon^{p\delta}\|a\|_{L^1}^p\int_{\Ri^d\backslash 4B}\Big[\f{1}{w(B(x_0,|x-x_0|))}\exp\Big(-\f{|x-x_0|^2}{c[\epsilon \rho_w(x,\mu)]^2}\Big)\Big]^pdw(x).
	\end{aligned}
	$$
	By Proposition \ref{crit}, we have
	\[
	\begin{aligned}
	\f{\epsilon\rho_w(x,\mu)}{|x-x_0|}&\lesi \f{\epsilon\rho_w(x_0,\mu)}{|x-x_0|}\Big(1+\f{|x-x_0|}{\rho_w(x_0,\mu)}\Big)^{\f{k_0}{k_0+1}}\\
	&\lesi \f{\epsilon\rho_w(x_0,\mu)}{|x-x_0|}\Big(1+\f{|x-x_0|}{\epsilon\rho_w(x_0,\mu)}\Big)^{\f{k_0}{k_0+1}}\\
	&\lesi \min\Big\{\f{\epsilon\rho_w(x_0,\mu)}{|x-x_0|}, \Big(\f{\epsilon\rho_w(x_0,\mu)}{|x-x_0|}\Big)^{\f{1}{k_0+1}}\Big\},
	\end{aligned}
	\]
	which implies that for any $N>0$ there exists a $C_N$ such that
	\begin{equation}
	\label{eq-inequlity e rho}
	\exp\Big(-\f{|x-x_0|^2}{c[\epsilon \rho_w(x,\mu)]^2}\Big)\le C_N \Big(\f{\epsilon\rho_w(x_0,\mu)}{|x-x_0|}\Big)^N.
	\end{equation}
	
	Inserting this into the above bound of $I_2$, we obtain, for $N>n(1-p)/p$, that
	$$
	\begin{aligned}
	I_2&\lesi \epsilon^{p\delta}w(B)^{p-1}\int_{\Ri^d\backslash 4B}\Big[\f{1}{w(B(x_0,|x-x_0|))}\Big(\f{\epsilon \rho_w(x_0,\mu)}{|x-x_0|}\Big)^N\Big]^pdw(x)\\
	&\lesi \epsilon^{p\delta}w(B)^{p-1}\int_{\Ri^d\backslash 4B}\Big[\f{1}{w(B(x_0,|x-x_0|))}\Big(\f{r}{|x-x_0|}\Big)^N\Big]^pdw(x)\\
	&\lesi \epsilon^{p\delta}.
	\end{aligned}
	$$
	
	\textbf{Case 2: $r< \epsilon\rho_w(x_0,\mu)/4$}.
	
	In this situation, since $\displaystyle \int a(y)dw(y)=0$, we have
	$$
	I_2=\int_{\Ri^d\backslash 4B}\Big|\sup_{0< t\leq [\epsilon \rho_w(x,\mu)]^2}\int_B (q_t(x,y)-q_t(x,x_0))a(y)dw(y)\Big|^pdw(x).
	$$
Owing to Proposition \ref{Proposition: holder of qt},
	$$
	\begin{aligned}
	I_2&=\int_{\Ri^d\backslash 4B}\Big[\sup_{0< t\leq [\epsilon \rho_w(x,\mu)]^2}\int_B \Big(\f{|y-x_0|}{\rho_w(x,\mu)}\Big)^{\theta} \f{1}{w(B(x,\sqrt{t}))}\exp\Big(-\f{|x-y|^2}{ct}\Big)|a(y)|dw(y)\Big]^pdw(x)\\
	&\lesi \int_{\Ri^d\backslash 4B}\Big[\int_B \Big(\f{r}{\rho_w(x,\mu)}\Big)^{\theta} \f{1}{w(B(x_0,|x-x_0|))}\exp\Big(-\f{|x-x_0|^2}{c[\epsilon\rho_w(x,\mu)]^2}\Big)|a(y)|dw(y)\Big]^pdw(x)\\
	&\lesi \int_{\Ri^d\backslash 4B}\Big[\int_B \Big(\f{\epsilon r}{|x-x_0|}\Big)^{\theta} \f{1}{w(B(x_0,|x-x_0|))}|a(y)|dw(y)\Big]^pdw(x)\\
	&\lesi \epsilon^{\theta p},
	\end{aligned}
	$$
	as long as $p>n/(n+\theta)$.
	This completes our proof.
\end{proof}

We are ready for the proof of Step 2.
\begin{proof}[\textbf{Proof of Step 2}:]
	Observe that for fixed numbers $\epsilon_1, \epsilon_2\in (0,1]$, there exists a $C=C(\epsilon_1,\epsilon_2)$ such that
	\begin{equation}\label{eq0-mainthm}
	C^{-1}\|\cdot\|_{h^{p,q}_{at,\rho_w,\epsilon_1}(\Ri^d,w)}\leq \|\cdot\|_{h^{p,q}_{at,\rho_w,\epsilon_2}(\Ri^d,w)} \leq C\|\cdot\|_{h^{p,q}_{at,\rho_w,\epsilon_1}(\Ri^d,w)}.
	\end{equation}
	For this reason, we need only to prove that  there exists an $\epsilon_0\in (0,1]$ so that
	\begin{equation}\label{eq1-mainthm}
	\|f\|_{h^{p,q}_{at,\rho_w,\epsilon_0}(\Ri^d,w)}\lesi \|f\|_{H^p_{L}(\Ri^d,w)}, \ \ \ \ f\in H^p_{L}(\Ri^d,w)\cap L_w^2(\Ri^d).
	\end{equation}

	Let $\{B_j\}_{j\in \mathbb{N}}$ and $\{\psi_j\}_{j\in \mathbb{N}}$ be families of balls and functions in Lemma \ref{xalpha psialpha}. For each $j\in \mathbb{N}$, $f\psi_j$ is supported in the ball $B_j=B(x_j,\rho_w(x_j,\mu))$ and $f\psi_j\in L^2_w(\Ri^d)$, which implies that $f\in h^p_{\ell}(\Ri^d,w)$ with $\ell =\rho_w(x_j,\mu)$. Applying Theorem \ref{thm1-local hardy spaces basic property}, we can decompose $f\psi_j$ into an atomic $(p,q,\rho_w,\epsilon)$-representation with $(p,q,\rho_w,\epsilon)$-atoms supported in $B^*_j$. Moreover we have, from Proposition \ref{crit}, the existence of $A_0$ so that 
	\[
	A_0^{-1}\rho_w(x_j,\mu)\leq  \rho_w(x,\mu)\leq A_0\rho_w(x_j,\mu) \ \ \ \text{for all $x\in B_j^*$ and all $ j\in\mathbb{N}$}.
	\] 
	This, in combination with Theorem \ref{thm- H rho and H rad}, yields
	\begin{equation*}
	\begin{aligned}
	\sum_{j\in \mathbb N}\|\psi_j f\|^p_{h^{p,q}_{at,\rho_w,\epsilon}(\Ri^d)}&\lesi \sum_{j\in \mathbb N}\Big\|\sup_{0<t<[\epsilon \rho_w(x_j,\mu)]^2}|e^{-tL_0}(\psi_j f)|\,\Big\|^p_{L^p_w(\Ri^d)}\\
	&\lesi \sum_{j\in \mathbb N}\Big\|\sup_{0<t<[\epsilon \rho_w(x_j,\mu)]^2}|e^{-tL_0}(\psi_j f)|\,\Big\|^p_{L^p_w(B_j^*)}\\
	& \ \ \ \ +\sum_{j\in \mathbb N}\Big\|\sup_{0<t<[\epsilon \rho_w(x_j,\mu)]^2}|e^{-tL_0}(\psi_j f)|\,\Big\|^p_{L^p_w(\Ri^d\backslash B_j^*)}.
	\end{aligned}
	\end{equation*}
	By Lemma \ref{Cor1: Hardy spaces},
	\[
	\sum_{j\in \mathbb N}\Big\|\sup_{0<t<[\epsilon \rho_w(x_j,\mu)]^2}|e^{-tL_0}(\psi_j f)|\,\Big\|^p_{L^p_w(\Ri^d\backslash B_j^*)}\lesi \epsilon^{\theta}\sum_{j\in \mathbb N}\|\psi_j f\|^p_{h^{p,q}_{at,\rho_w,\epsilon}(\Ri^d,w)}.
	\]
	By taking $\epsilon$ small enough, from these two estimates we infer
	\[
	\sum_{j\in \mathbb N}\|\psi_j f\|^p_{h^{p,q}_{at,\rho_w,\epsilon}(\Ri^d)}\lesi \sum_{j\in \mathbb N}\Big\|\sup_{0<t<[\epsilon \rho_w(x_j,\mu)]^2}|e^{-tL_0}\psi_j f|\,\Big\|^p_{L^p_w(B_j^*)}.
	\]
	From Proposition \ref{crit}, we can find $\tilde a_0$ such that
	\[
	\rho_w(x,\mu)\le \tilde a_0\rho_w(x_j,\mu)
	\]
	for all $x\in B^*_j$.
	
	Hence, setting $\tilde \epsilon =\tilde a_0 \epsilon$, then we have
	\begin{equation*}
	\begin{aligned}
	\sum_{j\in \mathbb N}\|\psi_j f\|^p_{h^{p,q}_{at,\rho_w,\epsilon}(\Ri^d)}	&\lesi \sum_{j\in \mathbb N}\Big\|\sup_{0<t<[\widetilde{\epsilon} \rho_w(\cdot,\mu)]^2}|e^{-tL_0}(\psi_j f)(\cdot)|\,\Big\|^p_{L^p_w(B_j^*)},
	\end{aligned}
	\end{equation*}
	which implies, thanks to Lemma \ref{xalpha psialpha},	 that
	\[
	\|f\|^p_{h^{p,q}_{at,\rho_w,\epsilon}(\Ri^d)}	\lesi \sum_{j\in \mathbb N}\Big\|\sup_{0<t<[\widetilde{\epsilon} \rho_w(\cdot,\mu)]^2}|e^{-tL_0}(\psi_j f)(\cdot)|\,\Big\|^p_{L^p_w(B_j^*)}.
	\]

	Therefore,
	\begin{equation*}
	\begin{aligned}
	\|f\|^p_{h^{p,q}_{at,\rho_w,\epsilon}(\Ri^d)}&\lesi  \sum_{j}\Big\|\sup_{0<t<[\widetilde{\epsilon} \rho_w(\cdot,\mu)]^2}|e^{-tL_0}(\psi_j f)(\cdot)-\psi_j(\cdot)e^{-tL_0} f(\cdot)|\,\Big\|^p_{L^p_w(B_j^*)}\\
	& \ \ \ +\sum_{j}\Big\|\sup_{0<t<[\widetilde{\epsilon} \rho_w(\cdot,\mu)]^2}|\psi_j(\cdot)[e^{-tL_0}-e^{-tL} ]f(\cdot)|\,\Big\|^p_{L^p_w(B_j^*)}\\
	& \ \ \ + \sum_{j}\Big\|\sup_{0<t<[\widetilde{\epsilon} \rho_w(\cdot,\mu)]^2}|\psi_j(\cdot)e^{-tL}f(\cdot)|\,\Big\|^p_{L^p_w(B_j^*)}.
	\end{aligned}
	\end{equation*}
	We now estimate the terms on the RHS of the inequality above.
	
	First, using Lemma \ref{xalpha psialpha}, 
	\[
	\begin{aligned}
	\sum_{j}\Big\|\sup_{0<t<[\widetilde{\epsilon} \rho_w(\cdot,\mu)]^2}|e^{-tL_0}(\psi_j f)(\cdot)&-\psi_j(\cdot)e^{-tL_0} f(\cdot)|\,\Big\|^p_{L^p_w(B_j^*)}\\&
	\lesi \Big\|\sum_{j}\sup_{0<t<[\widetilde{\epsilon} \rho_w(\cdot,\mu)]^2}|e^{-tL_0}(\psi_j f)(\cdot)-\psi_j(\cdot)e^{-tL_0} f(\cdot)|\,\Big\|^p_{L^p_w(\Ri^d)}\\
	&\lesi \epsilon^{\theta}\|f\|^p_{h^{p,q}_{at,\rho_w,\epsilon}(\Ri^d,w)},
	\end{aligned}
	\]
	where in the last step we used Lemma \ref{Lem2: Hardy spaces}.
	
	Secondly, by Lemma \ref{Lem3: Hardy spaces},
	\begin{equation*}
	\begin{aligned}
	\sum_{j}\Big\|\sup_{0<t<[\widetilde{\epsilon} \rho_w(\cdot,\mu)]^2}|\psi_j(\cdot)[e^{-tL_0}-e^{-tL} ]f(\cdot)|\,\Big\|^p_{L^p_w(B_j^*)}&\lesi \sum_{j}\Big\|\sup_{0<t<[\widetilde{\epsilon} \rho_w(\cdot,\mu)]^2}|[e^{-tL_0}-e^{-tL}]f|\,\Big\|^p_{L^p_w(\Ri^d)}\\
	&\lesi\epsilon^{\theta}\|f\|^p_{h^{p,q}_{at,\rho_w,\epsilon}(\Ri^d,w)}.
	\end{aligned}
	\end{equation*}
	
	Finally, by Proposition \ref{xalpha psialpha} again,
	\[
	\sum_{j}\Big\|\sup_{0<t<[\widetilde{\epsilon} \rho_w(\cdot,\mu)]^2}|\psi_j(\cdot)e^{-tL}f(\cdot)|\,\Big\|^p_{L^p_w(B_j^*)}\lesi \sum_{j}\Big\|\sup_{0<t<[\widetilde{\epsilon} \rho_w(\cdot,\mu)]^2}|e^{-tL}f|\,\Big\|^p_{L^p_w(\Ri^d)}.
	\]
	Consequently,
	\begin{equation}\label{eq2-lem3Hardy}
	\begin{aligned}
	\|f\|^p_{h^{p,q}_{at,\rho_w,\epsilon}(\Ri^d,w)}&\lesi \epsilon^{\theta}\|f\|^p_{h^{p,q}_{at,\rho_w,\widetilde{\epsilon}}(\Ri^d,w)}+\Big\|\sup_{0<t<[\widetilde{\epsilon} \rho_w(\cdot,\mu)]^2}|e^{-tL}f(\cdot)|\,\Big\|^p_{L^p_w(\Ri^d)}\\
	&\lesi \epsilon^{\theta}\|f\|^p_{h^{p,q}_{at,\rho_w,\widetilde{\epsilon}}(\Ri^d,w)}+\Big\|\sup_{t>0}|e^{-tL}f(\cdot)|\,\Big\|^p_{L^p_w(\Ri^d)}.
	\end{aligned}
	\end{equation}
    Taking $\epsilon$ small enough, 
    \[
    \|f\|^p_{h^{p,q}_{at,\rho_w,\epsilon}(\Ri^d,w)}\lesi \Big\|\sup_{t>0}|e^{-tL}f(\cdot)|\,\Big\|^p_{L^p_w(\Ri^d)}.
    \]
	
	This completes our proof.
\end{proof}

{\bf Acknowledgement.} The first-named author was supported by the research grant ARC DP170101060 from the Australian Research Council. The second-named and the third-named authors wish to express their sincere thanks to the support given by Vietnam’s
National Foundation for Science and Technology Development (NAFOSTED) under Project 101.02-2020.17.

\end{document}